\documentclass{amsart}
\usepackage[margin=1in]{geometry}
\usepackage{amsmath, amsfonts, amssymb, amsthm}
\usepackage{hyperref}
\usepackage{color}
\usepackage{dsfont}

\newtheorem{thm}{Theorem}[section]
\newtheorem{prop}[thm]{Proposition}

\newtheorem{conjecture}[thm]{Conjecture}

\newtheorem{lemma}[thm]{Lemma}
\newtheorem{defn}[thm]{Definition}
\newtheorem{preremark}[thm]{Remark}

\numberwithin{equation}{section}

\newcommand{\R}{\mathbb R}

\newcommand{\I}{\mathrm I}

\newcommand{\eps}{\varepsilon}
\newcommand{\dd} {\; \mathrm{d}}
\newcommand{\vv}{\langle v \rangle}

\DeclareMathOperator{\tr}{tr}

\title{Regularity estimates and open problems in kinetic equations}
\author{Luis Silvestre}
\address[L.~Silvestre]{Mathematics Department, University of Chicago,
  Chicago, Illinois 60637, USA} \email{luis@math.uchicago.edu}

\thanks{Luis Silvestre is supported by NSF grants 2054888 and 1764285.}

\begin{document}

\begin{abstract}
We survey some new results regarding a priori regularity estimates for the Boltzmann and Landau equations conditional to the boundedness of the associated macroscopic quantities. We also discuss some open problems in the area. In particular, we describe some ideas related to the global well posedness problem for the space-homogeneous Landau equation for Coulomb potentials.
\end{abstract}
\maketitle

\section{Introduction}
In this note, we discuss some recent developments and open problems concerning regularity estimates for kinetic equations. The material is based on the minicourse given by the author in the Barret Memorial lectures in May 2021. The lectures focused on the conditional regularity estimates for the Boltzmann equation. In this notes, we also discuss some additional topics. We describe the conditional regularity estimates for the Boltzmann and Landau equations, and the main techniques leading to that result. We explain how some techniques that were initially explored in the context of elliptic nonlocal equations apply to the context of the Boltzmann equation. We also discuss open problems in the space-homogeneous setting, and present some ideas and conjectures related to those problems.

The program on conditional regularity consists in studying solutions to the inhomogeneous Boltzmann or Landau equation under the assumption that their associated macroscopic hydrodynamic quantities stay pointwise bounded. This assumption rules out the formation of an implosion singularity that would be visible macroscopically. The conditional regularity estimates tell us, essentially, that no other type of singularity other than macroscopic implosions are possible for the non-cutoff Boltzmann or Landau equations.

There are various tools that play a role in obtaining the conditional regularity result. The first step, which is about $L^\infty$ bounds works for the case of hard or moderately soft potentials. In the very soft potential case (which corresponds to $\gamma+2s<0$, where $\gamma$ and $s$ are certain parameters in the collision kernel $B$ defined below), the lower order reaction term in the collision operator is too singular and difficult to control with the diffusion term. Currently, there are no upper bounds in the very soft potential case even for space-homogeneous solutions. The most extreme case, and also arguably the most interesting, is the Landau-Coulomb equation. We discuss the well known open problem of existence of global-in-time solutions for the Landau-Coulomb equation in Section \ref{s:space-homogeneous}. In this section we present informally some ideas that lead to new points of view on the problem. We include some conjectures.

We describe the regularity results including minimal technical details about their proofs. A more detailed description of the methods involved in obtaining the conditional regularity result for the Boltzmann equations can be found, still in survey form, in \cite{imbert-silvestre-survey2020}.

\section{Description of the equations}

Kinetic equations describe the evolution in time of a function $f(t,x,v)$ representing the density of particles in a dilute gas with respect to their position and velocity. The kinetic representation of a fluid lies at an intermediate scale between the large dimensional dynamical system following the position and velocity of each and every particle, and the macroscopic hydrodynamic description given by the Euler and Navier-Stokes equations.

A kinetic equation typically has the following form
\[ f_t + v \cdot \nabla_x f + F(t,x) \cdot \nabla_v f = Q(f,f).\]

The first two terms on the left hand side form the pure transport equation. In absence of any force, each particle moves at a straight line with its corresponding velocity $v$. A model in which there is no interaction between particles would simply be $f_t + v \cdot \nabla_x f = 0$. The third term, $F(t,x) \cdot \nabla_v f$, accounts for the macroscopic force. The force $F(t,x)$ may be the result of external forces, or a mean field potential force computed in terms of the distribution $f$ itself. The right hand side, $Q(f,f)$, accounts for local interactions between particles, typically in the form of collisions.

In these notes, we concentrate in models that are limited to local interactions. We take $F \equiv 0$ and reduce the equation to
\begin{equation} \label{e:boltzmann}
f_t + v \cdot \nabla_x f = Q(f,f).
\end{equation}

The right hand side $Q(f,f)$ is a quadratic nonlocal operator acting on $f(t,x,\cdot)$, for each value of $(t,x)$. The Boltzmann collision operator has the following form
\begin{equation} \label{e:bco}
Q(f,f)(v) = \int_{\R^d} \int_{S^{d-1}} (f'_\ast f' - f_\ast f) B(|v-v_\ast|,\cos \theta) \dd \sigma \dd v_\ast.
\end{equation}
Let us make some clarifications to understand the expression for $Q(f,f)$. It is a double integral, with respect to two parameters $v_\ast \in \R^d$ and a spherical variable $\sigma \in S^{d-1}$. The notation $f$, $f_\ast$, $f'$ and $f'_\ast$ denotes the values of the function $f$ evaluated at the same value of $(t,x)$, but with different velocities $v$, $v_\ast$, $v'$ and $v'_\ast$ respectively. The last two values depend on the integration parameters through the formulas
\begin{align*}
v' &= \frac{v+v_\ast}2 + \frac{|v-v_\ast|}2 \sigma, \\
v'_\ast &= \frac{v+v_\ast}2 - \frac{|v-v_\ast|}2 \sigma.
\end{align*}

First term in the integral, $f'_\ast f'$, accounts for particles with pre-collisional velocities $v'_\ast$ and $v'$ that may be colliding at the point $(t,x)$ and turning their velocities to $v$ and $v_\ast$. The loss term $-f_\ast f$ accounts for particles of velocities $v_\ast$ and $v$ that are colliding and changing their velocities. The rate by which these collisions occur is measured by the nonnegative kernel $B$. There are different choices for this kernel depending on modelling assumptions.

The angle $\theta$ measures the deviation between the pre and postcollisional velocities. It is precisely given by the formula
\[ \cos(\theta) = \frac{(v'-v'_\ast) \cdot (v-v_\ast)}{|v-v_\ast|^2}.\]
We write the collision kernel as $B(|v-v_\ast|,\cos \theta)$ to enphasize that it only depends on these two values. In particular, it is an even and $2\pi$-periodic functions of $\theta$.

The Boltzmann collision operator \eqref{e:bco}, for any kernel $B \geq 0$ and any function $f \geq 0$, has remarkable cacellation properties that we list here
\begin{itemize}
  \item $\int Q(f,f) \dd v = 0.$
  \item $\int v Q(f,f) \dd v = 0.$
  \item $\int |v|^2 Q(f,f) \dd v = 0.$
  \item $\int Q(f,f) \log f \,\dd v \leq 0.$
\end{itemize}
Each of these items corresponds to a conserved or monotone physical quantity. Any solution $f$ of the Boltzmann equation \eqref{e:boltzmann} formally satisfies the following conservation laws
\begin{itemize}
  \item Conservation of mass: the integral $\iint f(t,x,v) \dd v \dd x$ is constant in time.
  \item Conservation of momentum: the integral $\iint v \, f(t,x,v) \dd v \dd x$ is constant in time.
  \item Conservation of energy: the integral $\iint |v|^2 \, f(t,x,v) \dd v \dd x$ is constant in time.
  \item Entropy dissipation: the integral $\iint f(t,x,v) \log f(t,x,v) \dd v \dd x$ is monotone decreasing in time.
\end{itemize}

It can be verified directly that the only functions $f : \R^d \to [0,\infty)$ so that $Q(f,f)=0$ are the Gaussians. In other words, $Q(f,f)=0$ if and only if $f(v) = a \exp(-b |v-u|^2)$ for some $a,b \geq 0$ and $u \in \R^d$. In the kinetic context, these stationary solutions of \eqref{e:boltzmann} are called \emph{Maxwellians}.

\subsection{Common collision kernels and variants}

The exact form of the collision kernel $B$ depends on modelling choices. Particles that bounce against each other like billiard balls lead to a different collision kernel that if we modeled particles that behave like sponges. There are many different possible microscopic interactions between particles, leading to a variety of choices for the collision kernel $B$. Rigorously justifying the derivation of the Boltzmann equation from microscopic particle interactions is a delicate mathematical problem, with several fundamental open questions. In many cases, one can compute informal derivations based on euristics. In this section, we merely list the most common collision kernels $B$ that are found in the literature, without any attempt to justify them.

The hard spheres model is the one that results from particles that bounce like billiard balls. Its collision kernel $B$ takes a particularly simple form
\[ B(r,\cos \theta) = r.\]

Another natural model is to consider point particles that repell each other by a potential of the form $1/r^q$, with $q \geq 1$. In that case, the kernel $B$ is not explicit, but it satisfies the following bounds
\begin{equation} \label{e:non-cutoff}
B(r,\cos \theta) \approx r^\gamma |\sin(\theta/2)|^{-d+1-2s}.
\end{equation}
Here, we write $a \approx b$ to denote the fact that there is a constant $C$ so that $a/C \leq b \leq Ca$. The parameters $\gamma$ and $s$ depend on the power $q$ by the formulas
\[ \gamma=\frac{q-2d+2}{q}, \qquad 2s = \frac 2 {q}.\]
These kernels have a singularity at $\theta = 0$. Remarkably, for any values of $v$ and $v_\ast$
\[ \int_{S^{d-1}} B(|v-v_\ast|,\cos(\theta)) \dd \sigma = \infty.\]
The expression \eqref{e:bco} still makes sense for any smooth function $f$ when $s \in (0,1)$. This is because the cancellation in the factor $(f'_\ast f' - f_\ast f)$ compensates the singularity in $B$. It is the same way we normally make sense of general integro-differential operators with singular kernels. Indeed, as we will see in Section \ref{s:ide}, the Boltzmann collision operator consists of a diffusion term of fractional order $2s$ plus a lower order term.

When $s=1$, the kernel $B$ is too singular to make sense of the integral in \eqref{e:bco}. Taking $B(r,\cos(\theta)) = c_d (1-s) r^\gamma \sin(\theta/2)^{-d+1-2s}$, for some dimensional constant $c_d$, the expression \eqref{e:bco} converges to the following expression as $s \to 1$.
\begin{equation} \label{e:landau_expression1}
Q(f,f)(v) = \bar a_{ij}(v) \partial_{ij} f(v) + \bar c(v) f(v).
\end{equation}
Here
\begin{equation} \label{e:landau}
\bar a_{ij} = \int_{\R^d} (|w|^2 \delta_{ij} - w_i w_j) |w|^\gamma f(v-w) \dd w, \qquad \bar c = \partial_{ij} \bar a_{ij} = c f \ast |\cdot|^\gamma.
\end{equation}
The case of Coulombic potentials in 3D corresponds to $d=3$, $q=1$, $s=1$ and $\gamma=-3$. In that case, the convolution in \eqref{e:bco} also becomes too singular. An appropriate asymptotic limit \eqref{e:landau} as $\gamma \to -3$ leads to the Landau-Coulomb operator with
\begin{equation} \label{e:landau-coulomb}
\bar a_{ij} =  \partial_{ij} (-\Delta)^{-2} f = \frac 1 {8\pi} \int_{\R^3} (|w|^2 \delta_{ij} - w_i w_j) |w|^{-3} f(v-w) \dd w, \qquad \bar c = f.
\end{equation}

There are alternative equivalent expressions to \eqref{e:landau_expression1} that are useful depending on the type of computation that we want to do. The following integral expression is equivalent to \eqref{e:landau_expression1} and is useful, for example, to prove the dissipation of entropy.
\[ Q(f,f) = \partial_i \int_{\R^d} (|w|^2 \delta_{ij} - w_i w_j) |w|^\gamma ( f(v-w) \partial_j f(v) - \partial_j f(v-w) f(v) ) \dd w. \]

The Landau operator is the limit of the Boltzmann collision operator as its angular singularity approaches its limit point $s=1$. One could argue that because it is the most singular case, then the Landau operator should be \emph{harder} to study than the Boltzmann collision operator. However, the Landau operator involves only classical differentiation instead of an integro-differential expression. At the end, depending on what we want to prove, sometimes the Landau equation is harder, and sometimes it is simpler, than the Boltzmann equation.

It is useful to analyze the Landau operator as a way to understand the nature of the Boltzmann collision operator as well. The first term is a diffusion operator whose coefficients depend on the solution $f$. It is a quasilinear second order operator. The second term is of lower order. Note that the Landau operator is \textbf{not} the Laplacian plus a lower order term (i.e. it is not semilinear), in the same way that the Boltzmann collision operator is \textbf{not} the fractional Laplacian plus a lower order term.

\medskip

The Boltzmann collision operator \eqref{e:bco} does not make sense in one dimension ($d=1$). The constraints on the pre and postcollisional velocities force $\{v',v'_\ast\} = \{v,v_\ast\}$. There is a common one-dimensional toy model proposed by Kac in which there is conservation of energy but not of momentum. The Kac collision operator takes the following form.
\[ Q(f,f)(v) = \int_{\R} \int_{-\pi}^{\pi} (f'_\ast f' - f_\ast f) B(|v-v_\ast|,|\theta|) \dd \theta \dd v_\ast. \]
Here, we write
\begin{align*}
v' &= v \cos \theta - v_\ast \sin \theta, \\
v'_\ast &= v \sin \theta + v_\ast \cos \theta.
\end{align*}

In the original Kac model $B(r,|\theta|) = 1/2\pi$. It makes sense to consider non-cutoff kernels with the bounds \eqref{e:non-cutoff}. It is natural to expect that the regularity estimates that we prove for the Boltzmann equation in Theorem \ref{t:conditional-regularity_boltzmann} should apply to the Kac model as well. However, so far it has not been explicitly analyzed as far as we are aware.

\subsection{Mathematical problems in kinetic equations}

There are many mathematical problems related to kinetic equations. An important area or research is on the rigorous derivation of the models. We can study the derivation of every variant of the Boltzmann equation and the Landau equation from particle models. We can also study the derivation of the Euler and Navier-Stokes equation from the Boltzmann equation.

Another, rather different, direction of research is on the well posedness of the equations. Ultimately, the key to prove a well posedness result lies on the a priori estimates that we are able to obtain for the solution. Using merely the conservation of mass and the entropy dissipation, one can construct a very weak notion of global solution. They are the \emph{renormalized solutions} for the cutoff case, and the \emph{renormalized solutions with defect measure} for the non-cutoff case. The uniqueness of solutions within this class is currently unknown.

From an optimistic perspective, one may expect a solution $f$ to be $C^\infty$ smooth, strictly possitive everywhere, and with a strong decay as $|v| \to \infty$. Proving the existence of such a solution would require estimates that ensure these properties. The main focus of this note is on these a priori estimates.

There are several kind of estimates that have attracted people's attention through the years. The following is a rough attempt to classify them.
\begin{enumerate}
  \item \textbf{Moment estimates}. They refer to upper bounds on quantities of the form
  \[ \iint f(t,x,v) \omega(v) \dd v \dd x.\]
  Typically $\omega(v) = \langle v \rangle^q$. A moment estimate for some large value of $q$ should be understood as a type of decay estimate for $|v| \to \infty$.
  \item \textbf{Pointwise upper bounds}. They refer to an inequality of the form $f(t,x,v) \lesssim \omega(v)$. Typically, $\omega(v) = \langle v \rangle^{-q}$. One can easily argue that pointwise upper bounds are more desirable than moment estimates. Note that a bound of the form $f(t,x,v) \lesssim \langle v \rangle^{-d-q}$ implies a moment estimate for the weight $\langle v \rangle^p$ for any $p<q$.
  \item \textbf{Regularity estimates}. They may be upper bounds on some weighted Sobolev norm, on some H\"older norm, or an upper bound on the derivatives of $f$.
  \item \textbf{Lower bounds}. They quantify how far $f$ is from vacuum. They have the form $f \gtrsim \omega(v)$. Because the Maxwellians are stationary solutions, the best possible lower bound one can hope for would have a Gaussian decay as $v \to \infty$.
  \item \textbf{Convergence rates to equilibrium}. As time $t \to +\infty$, it is natural to expect that the solution $f$ will converge to the equilibrium Maxwellian distribution. In some cases, it is possible to quantify this rate of convergence.
\end{enumerate}

We use the Japanese bracket convention $\langle v \rangle := \sqrt{1+|v|^2}$.

Estimates of every kind described above are known in some circumstances. Below, we describe some of the mathematical difficulties and common simplifications.

\subsubsection{Cutoff vs non-cutoff}

It may seem convenient to be able to make sense of the formula \eqref{e:bco} for $Q(f,f)(v)$, even when $f$ is not necessarily smooth. For that, it would be desireable that the collision kernel $B$ is integrable with respect to the variable $\sigma$. The condition is known as Grad's cutoff assumption. It says that for any $v,v_\ast \in \R^d$,
\begin{equation} \label{e:cutoff}
\int_{S^{d-1}} B(|v-v_\ast|,\cos \theta) \dd \sigma < +\infty.
\end{equation}
The collision kernel $B$ corresponding to the hard spheres model satisfies this condition. The kernels $B$ that arise from inverse power law models \eqref{e:non-cutoff} do not satisfy \eqref{e:cutoff}. The Boltzmann equation with a kernel of the form \eqref{e:non-cutoff} would be referred to as the non-cutoff Boltzmann equation.

In the cutoff case, when \eqref{e:cutoff} holds, it is possible to split the Boltzmann operator $Q(f,f)$ into its \emph{gain} and \emph{loss} terms. That is $Q=Q_+ - Q_-$, where
\[ Q_+ = \int_{\R^d} \int_{S^{d-1}} f'_\ast f' \, B(|v-v_\ast|,\cos \theta) \dd \sigma \dd v_\ast, \qquad Q_- = \int_{\R^d} \int_{S^{d-1}} f_\ast f \, B(|v-v_\ast|,\cos \theta) \dd \sigma \dd v_\ast.\]
Both terms $Q_+$ and $Q_-$ are nonnegative. It is common to use this decomposition when studying upper and lower bounds for the cutoff Boltzmann equation.

In the non-cutoff case, it makes no sense to split $Q$ between a gain and a loss term. Indeed, the terms $Q_+$ and $Q_-$ defined above would be both equal to $+\infty$ in general.

Interestingly, the non-cutoff Boltzmann equation exhibits a regularization effect that does not hold in the cutoff case. The Boltzmann collision operator \eqref{e:bco} is in some sense a quasilinear integro-differential elliptic operator of order $2s$. This regularization effect is the reason why this note focuses on the non-cutoff case.

\subsection{Homogeneous vs inhomogeneous}

When we consider solutions to \eqref{e:boltzmann} that are constant in $x$, the equation is significantly simpler. The second term vanishes and we have
\[ f_t = Q(f,f).\]
Thinking that $Q(f,f)$ is a quasilinear elliptic operator, what we have is a nonlinear parabolic equation that satisfies
\begin{itemize}
  \item \textbf{Conservation of mass:} $\int_{\R^d} f(t,v) \dd v$.
  \item \textbf{Conservation of momentum:} $\int_{\R^d} v f(t,v) \dd v$.
  \item \textbf{Conservation of energy:} $\int_{\R^d} |v|^2 f(t,v) \dd v$.
  \item \textbf{Monotone decreasing entropy:} $\int_{\R^d} f(t,v) \log f(t,v) \dd v$.
\end{itemize}

From a macroscopic perspective, the space homogeneous equation is rather plain. It corresponds to a fluid whose macroscopically observable quantities are homogeneous and stationary. From a mathematical perspective, the homogeneous problem is significantly simpler than the inhomogeneous. The well posedness is well understood in the homogeneous non-cutoff Boltzmann equation when $\gamma+2s \geq 0$. Remarkably, the existence of global in time classical solutions when $\gamma+2s < 0$ remains open.

\section{On the existence of solutions}

For any evolution equation that comes from mathematical physics, we are interested in whether the problem is well posed or not. In the case of the Boltzmann equation, we consider the Cauchy problem for a given initial value $f_0(x,v) \geq 0$, and we would want to determine if there exist a solution $f$ to \eqref{e:boltzmann} so that $f(0,x,v) = f_0(x,v)$.

We know that a generalized notion of solution exists globally. It was established first in the cutoff case \cite{diperna1989} and then also for non-cutoff \cite{alexandre-villani2002renormalized}. It is unknown, for any kind of physically relevant collision kernel, whether these solutions may develop singularities in finite time starting from smooth data. It is unclear if the solution will stay unique past a potential singularity. Establishing the existence of global smooth solutions depends crucially on obtaining a priori regularity estimates.

Results about the existence of classical solutions can be classified in three rough sub-categories. They are the following.
\begin{enumerate}
  \item \textbf{Global well posedness near equilibrium.} If the initial value $f_0$ is sufficiently close to a Maxwellian with respect to some appropriate norm, then it is often possible to prove that there exists a smooth solution with $f_0$ as its initial data.
  \item \textbf{Short time existence results.} For generic values of $f_0$ in some functional space, a short-time existence result says that there exists a solution $f$ in some interval of time $[0,T]$ for $T$ sufficiently small (depending on $f_0$).
  \item \textbf{Unconditional global well posedness.} For generic values of $f_0$ in some functional space, we would want to prove that there exists a solution $f$ that exists forever. This is the ultimate goal, but it might also be too ambitious.
\end{enumerate}

The analysis of solutions for short-time is a basic type of well posedness. However, it is more complex issue than one might initially expect. The first result of this kind for the inhomogeneous non-cutoff Boltzmann equation appeared in \cite{amuxy2011qualitative} (see also \cite{amuxy-existence2}). It applies for $0<s<1/2$ and $\gamma+2s < 1$. It also requires the initial data to belong to a Sobolev space involving at least four derivatives in $x$ and $v$, with a weight that forces $f_0$ to have Gaussian decay as $v \to \infty$. In particular, the result in \cite{amuxy2011qualitative} cannot be applied for initial data that has polynomial, or even simply exponential decay, for large velocities.

The first result that allows initial data with algebraic decay rate appeared in \cite{morimoto2015polynomial}. It was later improved in \cite{henderson2020polynomial}. It applies in the case of soft potentials ($\gamma \leq 0$) and requires the initial data in $H^4$ with a polynomial weight.

As far as we are aware, there are no further short-time existence results for the inhomogeneous Boltzmann equation without cutoff explicitly covered in the literature. The case of initial data with algebraic decay rate (as in $f_0 \lesssim \langle v \rangle^{-q}$), with hard potentials ($\gamma > 0$) is not handled by any documented result. It is remarkable that even though the Boltzmann equation is one of the most studied problems in mathematical physics, such a basic question still remains unanswered. Note also that the issue of uniqueness of solutions depends exclusively on our short-time well posedness theory.

There are a number of papers dealing with solutions near equilibrium, for practically the full range of parameters $\gamma$ and $s$. See \cite{gressmanstrain2011,amuxy2011_hard,amuxy2012_soft,duan2019global,herau2020regularization,alonso2018non,alonso2020giorgi,zhang2020global,silvestre2021solutions}. From the proofs in any of these papers, in theory one should be able to extract a proof of short-time existence. Its conditions on the initial data would be milder than for the global existence result. However, in most of these papers the short-time existence results are not explicitly stated, and it is difficult to determine what can be done in each case.

For any inhomogeneous model like \eqref{e:boltzmann}, with a nontrivial and physically meaningful collision kernel $Q$, the global existence of classical solutions for generic initial data $f_0$ remains an outstanding open problem.

A short-time existence result is a first step in the Cauchy theory. In order to make the solution global, we would need sufficiently strong a priori estimates. If the solution stays bounded in a sufficiently strong norm, one would be able to reapply the short time existence result arbitrarily many times and extend the solution forever. The key for global well posedness would be to combine a short-time existence result with good enough a-priori estimates. It is not necessary to work with weak solutions. A priori estimates for classical solution would suffice, since these are the type of solutions provided by short-time existence results. We review a priori estimates for the inhomogeneous Boltzmann equation in these notes. As we will see, they are not strong enough to ensure the unconditional global well posedness of the equation yet.

While global well posedness results near equilibrium are interesting in themselves, they say hardly anything about the solutions for arbitrary initial data. Moreover, any appropriate short-time existence result combined with the convergence to equilibrium \cite{desvillettes2005global}, and Theorem \ref{t:conditional-regularity_boltzmann}, implies the existence of global solutions when the initial data is sufficiently close to a nonzero Maxwellian (see \cite{silvestre2021solutions}). It is also possible to study the global well posedness of the equations when the initial data in near zero (see \cite{luk2019} and \cite{chaturvedi2021}). It is a completely different setting, driven by dispersion rather than by diffusion.

\subsection{Conditional regularity estimates}

We now describe our recent results on a priori regularity estimates. Let us start with some discussion of what we would want to prove, and what we can realistically achieve.

The hydrodynamic equations, like Euler and Navier-Stokes, describe the evolution of a fluid at a coarser scale than kinetic equations. The velocity $u(t,x)$, density $\rho(t,x)$, and temperature $\theta(t,x)$, in the compressible Euler equation make sense at a macroscopic scale. Microscopically, they correspond to some local average over particles of the fluid around each point $(t,x)$. The compressible Euler equation can be obtained formally as an asymptotic limit of solutions to the Boltzmann equation. The compressible Euler equation develops singularities in finite time. Should we expect the Boltzmann equation to develop singularities as well? Or is there any reason to believe that the Boltzmann equation will act as a regularization of the usual hydrodynamic equations?

There are two very different kinds of singularities that one observes in the compressible Euler equation: shocks and implosions. A \emph{shock} is the same type of singularity that emerges from the flow of Burgers equation. The values of the function stay bounded, and a discontinuity is created when two characteristic curves collide. It is as if the equation was pushing the solution to take two contradictory values at the same point, resulting in a discontinuity. In the kinetic setting, on the contrary, there can be no such thing as a shock singularity. The function $f(t,x,v)$ allows a distribution of different velocities coexisting at each point $(t,x)$. An \emph{implosion} singularity has a very different nature. The flows concentrates at one point, where the energy concentrates and its density explodes. While some implosion profiles for the Euler equations appeared a long time ago, it was only recently that smooth implosion profiles and their stability were properly understood \cite{merle2019smooth,merle2019implosion}. There is no obvious reason to rule out the existence of implosion singularities for the Boltzmann equation. It seems conceivable that there may exist a solution to the Boltzmann equation \eqref{e:boltzmann} whose mass, momentum and temperature density behave similarly as in the implosion singularity of the compressible Euler equation. Rigorously constructing, or ruling out, such a solution seems like a very difficult problem right now.

Beyond the possibility of implosion singularities, we may still wonder if there exist any other kind of kinetic singularities for the Boltzmann equation that are not related to anything that we can see macroscopically in the Euler equation. Our conditional regularity results described below rule it out. Essentially, they say that the solution to the non-cutoff Boltzmann (and Landau) equation will be $C^\infty$ smooth for as long as there is no implosion singularity in its associated macroscopically observable quantities.

Let us state our conditions. They say that the mass density is bounded below and above, the energy density is bounded above, and the entropy density is bounded above. The following three inequalities will be assumed to hold at every point $(t,x)$.
\begin{align}
m_0 \leq \int_{\R^d} f(t,x,v) \dd v &\leq M_0,  \label{e:mass_density_assumption} \\
\int_{\R^d} f(t,x,v) |v|^2 \dd v &\leq E_0, \label{e:energy_density_assumption} \\
\int_{\R^d} f(t,x,v) \log f(t,x,v) \dd v &\leq \label{e:entropy_density_assumption} H_0.
\end{align}

We stress that we will not prove the hydrodynamic inequalities \eqref{e:mass_density_assumption}, \eqref{e:energy_density_assumption} and \eqref{e:entropy_density_assumption}. We take them as an assumption. From the discussion above, one would be inclined to believe that there may exist some initial data that evolves into an implosion singularity for which all these quantities blow up at a point. However, such an solution has not been constructed yet for any physically reasonable collision operator $Q$.

Here is our conditional regularity theorem proved in \cite{imbert2019global}.

\begin{thm} \label{t:conditional-regularity_boltzmann}
Let $f$ be a (classical) solution to the inhomogeneous Boltzmann equation \eqref{e:boltzmann}, periodic in space, where the collision operator $Q$ has the form \eqref{e:bco} and $B$ is the standard non-cutoff collision kernel of the form \eqref{e:non-cutoff} with $\gamma+2s \in [0,2]$.

Assume that there are constants $m_0>0$, $M_0$, $E_0$ and $H_0$ such that the hydrodynamic bounds \eqref{e:mass_density_assumption}, \eqref{e:energy_density_assumption} and \eqref{e:entropy_density_assumption} hold. Moreover, if $\gamma \leq 0$, we also assume that the initial data $f_0$ decays rapidly at infinity: $f_0 \leq C_q \langle v \rangle^{-q}$ for all $q>0$.

\medskip

Then, for $\tau>0$, all derivatives $D^\alpha f$, for any multi-index $\alpha$, are bounded in $[\tau,\infty) \times \R^d \times \R^d$ an decay as $|v| \to \infty$. More precisely, there are constants $C_{\tau, \alpha,q}$ so that
\[ \langle v \rangle^q |D^\alpha f(t,x,v)| \leq C_{\tau,\alpha,q} \qquad \text{for all } t>\tau, x\in \R^d, v \in \R^d.\]
\end{thm}

The proof of theorem \ref{t:conditional-regularity_boltzmann} contains several ingredients that originated in the study of nonlocal parabolic equations. We describe some of the key ingredients in this manuscript. A more detailed description of the proof, still in survey form, is given in \cite{imbert-silvestre-survey2020}.

There is a parallel result for the inhomogeneous Landau equation that is proved by Chris Henderson and Stan Snelson in \cite{hst-smoothing-landau}. We state it here.

\begin{thm} \label{t:conditional-regularity_landau}
Let $f$ be a (possibly weak) solution to the inhomogeneous Landau equation \eqref{e:boltzmann} where the collision operator $Q$ has the form \eqref{e:landau_expression1} and \eqref{e:landau} with $\gamma \in (-2,0)$.

Assume that there are constants $m_0>0$, $M_0$, $E_0$ and $H_0$ such that the hydrodynamic bounds \eqref{e:mass_density_assumption}, \eqref{e:energy_density_assumption} and \eqref{e:entropy_density_assumption} hold. There is a $\mu>0$, depending on $m_0$, $M_0$, $E_0$, $H_0$ and $\gamma$, so that if
\[ f_0(x,v) \leq C e^{-\mu |v|^2}\]

then, for any $\tau>0$, $\mu' \in (0,\mu)$, all derivatives $D^\alpha f$, for any multi-index $\alpha$, are bounded in $[\tau,\infty) \times \R^d \times \R^d$ an decay as $|v| \to \infty$. More precisely, there are constants $C_{\tau, \mu', \alpha,q}$ so that
\[ e^{\mu' |v|^2} |D^\alpha f(t,x,v)| \leq C_{\tau,\mu',\alpha,q} \qquad \text{for all } t>\tau, x\in \R^d, v \in \R^d.\]
\end{thm}

The proof of Theorem \ref{t:conditional-regularity_landau} contains similar steps as the proof of Theorem \ref{t:conditional-regularity_boltzmann}, except that instead of regularity techniques for nonlocal equations, it involves more classical second order diffusion. It uses the upper bounds from \cite{css-upperbounds-landau} instead of the ones from \cite{silvestre-boltzmann2016,imbert-mouhot-silvestre-decay2020}, and the Harnack inequality from \cite{gimv-harnacklandau} instead of \cite{imbert-silvestre-whi2020}.

The regularity estimates from Theorems \ref{t:conditional-regularity_boltzmann} and \ref{t:conditional-regularity_landau} stay uniform as $t \to \infty$, or for as long as \eqref{e:mass_density_assumption}, \eqref{e:energy_density_assumption} and \eqref{e:entropy_density_assumption} hold. This is important. For example, the celebrated convergence to equilibrium result by Desvillettes and Villani in \cite{desvillettes2005global} applies to solutions provided that they are uniformly $C^\infty$ as $t \to \infty$. However, if we only seek to construct a global smooth solution, it would suffice to have regularity estimates that are allowed to deteriorate for large time. In that case, it is shown in \cite{hst-continuation-landau,hst-continuation_boltzmann} that the upper bounds on mass density and energy density suffice. We summarize both results in the following statement.

\begin{thm} \label{t:continuation}
Let $f$ be a solution to \eqref{e:boltzmann}, periodic in space, where the collision operator $Q$ is either the non-cutoff Boltzmann collision operator \eqref{e:bco} with $B$ as in \eqref{e:non-cutoff} with $s \in (0,1)$ and $\gamma+2s \in [0,2]$, or the Landau operator \eqref{e:landau_expression1} and \eqref{e:landau} with $\gamma \in (-2,0)$.

Assume that there are upper bounds for the mass and energy density. That means that there is a constant $N_0$ so that for all $t,x$,
\begin{align*}
\int f(t,x,v) (1+|v|^2) \dd v &\leq N_0.
\end{align*}

Assume that $f_0$ is continuous and periodic in $x$. It does not need to be strictly positive. There might be values of $x$ so that $f_0(x,\cdot)$ vanishes.

In the case that $\gamma \leq 0$, for the Boltzmann equation, we also assume that $f_0(x,v) \leq C_q \langle v \rangle^{-q}$ for all $q>0$.

In the case of the Landau equation, we also assume $f_0(x,v) \leq C e^{-\mu |v|^2}$.

Then, for all derivatives $D^\alpha f$, for any multi-index $\alpha$, there exist upper bounds of the form
\[ |D^\alpha f(t,x,v)| \leq \begin{cases}
C_{\alpha,q,f_0}(t) \vv^{-q} & \text{for any } q \geq 0, \text{ for Boltzmann,} \\
C_{\alpha,\mu',f_0}(t) e^{-\mu' |v|^2} & \text{for any } \mu' \in [0,\mu), \text{ for Landau}. \\
\end{cases}
\]
\end{thm}
Unlike the upper bounds in Theorems \ref{t:conditional-regularity_boltzmann} and \ref{t:conditional-regularity_landau}, the upper bounds in Theorem \ref{t:continuation} deteriorate as $t \to \infty$. They suffice to construct a global smooth solution, but they would not let us apply the result in \cite{desvillettes2005global} and deduce their relaxation to equilibrium.

Theorem \ref{t:continuation} builds upon Theorems \ref{t:conditional-regularity_boltzmann} and \ref{t:conditional-regularity_landau} by propagating a lower bound that is stronger than the first inequality in \eqref{e:mass_density_assumption}. Indeed, for any nonzero continuous initial data $f_0$, there will be a clear ball $B = B_r(x_0,v_0)$ and $\delta>0$ so that $f_0(x,v) \geq \delta$ whenever $(x,v) \in B$. The upper bounds on mass and energy densities turn out to suffice to propagate and expand this lower bound to the full space by a barrier-like argument. Thus, the lower bound in \eqref{e:mass_density_assumption} holds for positive time. The upper bound on the entropy \eqref{e:entropy_density_assumption} is not proved directly. Instead, the authors observe that the proofs of Theorems \ref{t:conditional-regularity_boltzmann} and \ref{t:conditional-regularity_landau} go through without \eqref{e:entropy_density_assumption} as soon as we have a barrier from below. The right hand side of the regularity estimates in Theorem \ref{t:continuation} depend on the initial data in terms of this clear ball.

\section{Integro-differential diffusion inside the Boltzmann collision operator}
\label{s:ide}

The analysis of equations involving integro-differential diffusion has been an intensely researched area since the beginning of the 21st century. A linear \emph{elliptic} integro-differential operator has the general form
\begin{equation} \label{e:id-op}
Lf(v) = \int (f(v')- f(v)) K(v,v') \dd v',
\end{equation}
for some nonnegative kernel $K$. The typical linear parabolic integro-differential equation would be $f_t = Lf$, for an nonlocal operator $L$ as above.

The operator $L$ in \eqref{e:id-op} is the nonlocal analog of more classical second order elliptic operators. In the case the $K(v,v') = |v-v'|^{-d-2s}$, it is a multiple of the fractional Laplacian: $Lf(v) = -c (-\Delta)^sf(v)$. The operator $L$ would be considered elliptic of order $2s$ if its kernel is comparable with the kernel of the fractional Laplacian $|v-v'|^{-d-2s}$. However, this concept is more subtle than it may seem initially. There are several different ellipticity conditions in the literature depending on different interpretations of the word \emph{comparable}.

There are two types of second order elliptic operators: divergence-form operators $Lf = \partial_i a_{ij}(v) \partial_j f$, and nondivergence-form operators $Lf = a_{ij}(v) \partial_{ij} f$. In both cases, the uniform ellipticity condition consists in requiring the coefficients $a_{ij}(v)$ to be strictly positive definite at every point $v$, with some specific upper and lower bounds $\lambda I \leq \{a_{ij}\} \leq \Lambda I$. In the nonlocal setting \eqref{e:id-op}, the divergence or nondivergence structure of the operator will be reflected in different symmetry conditions on the kernel $K$. A possible uniform ellipticity condition would be to require $\lambda |v-v'|^{-d-2s} \leq K(v,v') \leq \Lambda |v-v'|^{-d-2s}$. It is a common condition in the literature of nonlocal equations. However, the richness of the class of integral kernels gives us a lot of flexibility, and more general notions of ellipticity exist.

Most of the central regularity results for elliptic and parabolic equations have a nonlocal counterpart. In particular, there are nonlocal versions of De Giorgi-Nash-Moser theorem \cite{komatsu1995, barlow2009non,basslevin2002,caffarelli2010drift,kassmann2009priori,Felsinger2013,chan2011,kassmann2013regularity}, Krylov-Safonov theorem \cite{basslevin2002,bass2002harnack,song2004,bass2005holder,bass2005harnack,%
silvestre2006holder,caffarelli2009regularity,silvestre2011differentiability,davila2012nonsymmetric,davila2014parabolic,kassmann2013intrinsic,bjorland2012,rang2013h,schwab2016}, and Schauder estimates \cite{MP,tj2015,serra2015,imbert2016schauder,tj2018}.

In order to apply methods that originate in the study of elliptic equations in divergence form (like in the theorem of De Giorgi, Nash and Moser), we would need the kernel $K$ to satisfy a symmetry condition that allows us to apply variational techniques. The usual elliptic operators in divergence form $f \mapsto \partial_i a_{ij} \partial_j f$ are self-adjoint. Analogously, the integro-differential operator $L$ above is self-adjoint when $K(v,v') = K(v',v)$.

Conversely, in order to apply methods that originate in the study of elliptic equations in non-divergence form (like in the theorem of Krylov and Safonov), we need different symmetry assumptions on the kernel $K$. The key property of non-divergence operators like $f \mapsto a_{ij} \partial_{ij} f$ is that for any smooth function $f$ the values of the operator make sense pointwise. This is always true for integro-differntial operators when $s<1/2$. If $s \geq 1/2$, we would need the symmetry condition $K(v,v+w) = K(v,v-w)$. The type of techniques that one can apply to integro-differential equations vary according to the symmetry conditions on the kernel.

The Boltzmann collision operator \eqref{e:bco} involves an integral of differences. However, it is not immediately clear how to relate it with results about nonlocal operators of the form \eqref{e:id-op}. With that in mind, let us analyze the expression for $Q(f,f)$ in \eqref{e:bco} and derive an equivalent formulation.

We add an subtract a term $f'_\ast f$ inside the integrand.
\begin{align*}
Q(f,f)(v) &= \int_{\R^d} \int_{S^{d-1}} (f'_\ast f' - f'_\ast f + f'_\ast f - f_\ast f) \, B \dd \sigma \dd v_\ast, \\
&= \int_{\R^d} \int_{S^{d-1}} (f' - f) \, [f'_\ast B] \dd \sigma \dd v_\ast + f(v) \int_{\R^d} \int_{S^{d-1}} (f'_\ast - f_\ast) B \dd \sigma \dd v_\ast
\end{align*}
For the first term, we use a change of variables known as \emph{Carleman coordinates}. For the second term, we refer to the \emph{cancellation lemma}.

The cancellation lemma appeared first in \cite{alexandre_entropy_dissipation2000} (see also \cite{silvestre-boltzmann2016}). It reduces the integral in the second term to a convolution.
\begin{lemma}[Cancellation lemma]
There is a nonnegative function $b : \R^d \to [0,\infty)$, depending only on the collision kernel $B$ on \eqref{e:bco}, so that
\[ \int_{\R^d} \int_{S^{d-1}} (f'_\ast - f_\ast) B \dd \sigma \dd v_\ast = [f \ast b](v)\]
In particular, if $B$ has the standard non-cutoff form \eqref{e:non-cutoff}, then $b(v) \approx |v|^\gamma$.
\end{lemma}

For the first term, we want to change the variables of integration from $v_\ast \in \R^d$ and $\sigma \in S^{d-1}$ to $v' \in \R^d$ and $w \in \langle v'-v \rangle^\perp$, so that $v'_\ast = v+w$ and $v_\ast = v'+w$. With these new variables, it can be checked (see \cite{silvestre-boltzmann2016}) that
\[ \dd \sigma \dd v_\ast = \frac{2^{2d-1}}{|v'-v| |v-v_\ast|^{d-2}}\dd w \dd v'. \]
Thus, the first term becomes
\[ \int_{\R^d} \int_{S^{d-1}} (f' - f) \, [f'_\ast B] \dd \sigma \dd v_\ast = \int_{\R^d} (f'-f) K_f(v,v') \dd v', \]
where
\begin{equation}\label{e:boltzmann-kernel}
K_f(v,v') = \frac{2^{2d-1}}{|v'-v|} \int_{w \perp v'-v}  |v-v_\ast|^{-d+2} f(v'_\ast) B(|v-v_\ast|,\cos \theta) \dd w.
\end{equation}

For a standard non-cutoff kernel $B$ satisfying \eqref{e:non-cutoff}, we obtain
\begin{equation}\label{e:boltzmann-kernel-approx}
K_f(v,v') \approx |v'-v|^{-d-2s} \int_{w \perp v'-v}  |w|^{\gamma+2s+1} f(v+w) \dd w.
\end{equation}
When this last integral factor is bounded below and above, we recover the most classical notion of classical ellipticity $K_f(v,'v) = |v-v'|^{-d-2s}$. However, the integral is on a hyperplane. If $f$ is a function in $L^1$, with finite second moment, it is not possible to bound this integral pointwise either from below or from above. For that reason, we must explore more general notions of ellipticity.

Let us suppose that all we know about the function $f : \R^d \to [0,\infty)$ is that it has finite mass (integral), energy (second moment) and entropy. Moreover, its mass is bounded below. We state these assumptions precisely below.
\begin{align*}
0 < m_0 \leq \int f(v) \dd v \leq M_0, \\
\int f(v) |v|^2 \leq E_0, \\
\int f(v) \log f(v) \dd v \leq H_0.
\end{align*}
Based on these four constants only, we can deduce the following nondegeneracy conditions on the kernel $K_f$ (see \cite{imbert-silvestre-survey2020,silvestre-boltzmann2016,imbert-silvestre-whi2020,imbert2019global}).
\begin{description}
\item[Symmetry in the \emph{nondivergence} way] This is automatic from the formula \eqref{e:boltzmann-kernel}.
\[ K(v,v+w) = K(v,v-w).\]
\item[Lower bound on a cone of directions] There is a constant $\lambda>0$ depending on $m_0$, $M_0$, $E_0$ and $H_0$ such that
\[K(v,v') \geq \lambda  |v'-v|^{-d-2s}, \]
when $v'$ belongs to a cone of directions emanating from $v$. This cone is \emph{thick} in the sense that the measure of its intersection with the unit sphere centered at $v$ has a lower bound depending also on $m_0$, $M_0$, $E_0$ and $H_0$.
\item[Upper bound on average] There is a constant $\Lambda$ depending only on $M_0$ and $E_0$ such that
\[ \int_{B_{r}(v)} K_f(v,v') |v-v'|^2 \dd v' \leq \Lambda r^{2-2s}.\]
\item[Cancellation] While the exact \emph{divergence-form} symmetry does not hold. We do have some symmetry in the form of cancellation estimates
\[ \left\vert \int_{B_r(v)} K_f(v,v') - K_f(v',v) \dd v' \right\vert \leq \Lambda r^{-2s}.\]
If $s \geq 1/2$, we also have
\[ \left\vert \int_{B_r(v)} (v-v') (K_f(v,v') - K_f(v',v)) \dd v' \right\vert \leq \Lambda r^{1-2s}.\]
The first of these inequalities is a rewriting of the classical cancellation lemma from \cite{alexandre_entropy_dissipation2000} in terms of the kernel $K_f$.
\end{description}

These are the estimates that we get for $K_f$ based on minimal physically meaningful assumptions on $f$. Interestingly, much of the theory of elliptic and parabolic integro-differential equations can be recovered for equations involving kernels that satisfy these assumptions only. In that sense, the lower bound on the cone of nondegeneracy, and the upper bound on average, described above, are some mild form of integro-differential uniform ellipticity of order $2s$.

The values of $\lambda$ and $\Lambda$, and the thickness of the cone of degeneracy, depend on $m_0$, $M_0$, $E_0$ and $H_0$, and also on $|v|$. Their values degenerate in certain precise way as $|v| \to \infty$.

A key property of divergence form elliptic operators that is central in the proof of the theorem of De Giorgi, Nash and Moser is their coercivity and boundedness with respect to the $H^1$ norm. In the case of integro-differential operators, we work with fractional order Sobolev spaces. To have a nonlocal analog of this important regularity theorem, we want the integro-differential operator $L$ to be bounded and coercive with respect to the $H^s$ norm. The following results tell us that the cone condition and the upper bound on average described above for the Boltzmann kernel $K_f$ are enough.

\begin{prop} \label{p:Hs_boundedness}
Assume that a nonnegative kernel $K: \R^d \times \R^d \to \R$ satisfies the following two conditions
\begin{itemize}
\item There is a constant $\Lambda$ such that
\[ \int_{B_{r}(v)} K(v,v') |v-v'|^2 \dd v' \leq \Lambda r^{2-2s}.\]
\item The cancellation condition holds
\[ \left\vert \int_{B_r(v)} K(v,v') - K(v',v) \dd v' \right\vert \leq \Lambda r^{-2s}.\]
If $s \geq 1/2$, we also have
\[ \left\vert \int_{B_r(v)} (v-v') (K(v,v') - K(v',v)) \dd v' \right\vert \leq \Lambda r^{1-2s}.\]
\end{itemize}
Then the operator $L$ defined in \eqref{e:id-op} is a well defined bounded linear operator from $H^s$ to $H^{-s}$.
\end{prop}

As we discussed, the kernel $K_f$ of the Boltzmann equation satisfies the assumption of Proposition \ref{p:Hs_boundedness}. Yet, Proposition \ref{p:Hs_boundedness} is a general statement of integro-differential operators. It applies to any kernel $K$, regardless of whether it was obtained from the Boltzmann collision operator, or from any other origin. A proof of Proposition \ref{p:Hs_boundedness} is given in \cite{imbert-silvestre-whi2020}.

A form of coercivity follows from the cone of nondegeneracy described above. It is a consequence of the following general result for integro-differential operators proved in \cite{chaker2020}

\begin{thm} \label{t:coercivity}
Let $K: B_2 \times B_2 \to \R$ be a nonnegative kernel. Assume that there exist two constants $\lambda>0$ and $\mu>0$ so that for any $v \in B_2$ and any ball $B \subset B_2$ that contains $v$, we have 
\[ |\{v' \in B : K(v,v') \geq \lambda |v-v'|^{-d-2s} \}| \geq \mu |B|. \]
Then, the following coercivity holds
\[ \iint_{B_2 \times B_2} |f(v') - f(v)|^2 K(v,v') \dd v' \dd v \geq c \lambda \iint_{B_1 \times B_1} |f(v') - f(v)|^2 |v-v'|^{-d-2s} \dd v' \dd v = \tilde c \lambda \|f\|_{\dot H^s}. \]
Here, the constants $c$ and $\tilde c$ depend on $\mu$, $s$ and dimension only.
\end{thm}

The cone of nondegeneracy condition satisfied by the Boltzmann kernel $K_f$ is naturally stronger than the assumption of Theorem \ref{t:coercivity}. The coercivity of the Boltzmann collision operator was proved several times in the literature using various techniques. It is interesting to realize that it follows as a consequence of a general statement of integro-differential quadratic forms.

\section{Regularity estimates for the Kinetic Fokker-Planck equation}

The collision operator $Q(f,f)$ in \eqref{e:boltzmann} acts as a diffusion in the velocity variable. It may be a nonlocal diffusion in the case of the Boltzmann equation, or a more classical second order diffusion in the case of Landau equation. In any case, we must address the fact that the equation has a regularization effect with respect to all variables $t$, $x$ and $v$, even though the diffusion is with respect to velocity only. It is a hypoelliptic effect that comes from the interaction between the diffusion in velocity and the kinetic transport.

In 1934, Kolmogorov studied the following equation
\[ f_t + v \cdot \nabla_x f = \Delta_v f. \]
It is a simpler linear model than \eqref{e:boltzmann}, where the collision operator $Q$ is replaced by the usual Laplacian. Kolmogorov observed that for any initial data $f_0 \in L^1 + L^\infty$, the equation has a smooth solution. It follows immediately after explicitly computing its fundamental solution.
\[ K(t,x,v) = \begin{cases}
c_d t^{-2d} \exp \left( -\frac{|v|^2}{4t} - \frac{3 |x-tv/2|^2}{t^3} \right) &\text{for } t>0, \\
0 &\text{for } t \leq 0.
\end{cases}
\]
We compute the solution $f$, for any initial data $f_0$, by a modified convolution of $f_0$ with $K$. From this formula, we observe immediately that the solution is $C^\infty$ smooth in all variables.

For fractional diffusion, a similar analysis applies, but this time the fundamental solution is not explicit. If we consider, for $s \in (0,1),$
\[ f_t + v \cdot \nabla_x f + (-\Delta_v)^s f = 0, \]
there is a fundamental solution $K_s$ whose Fourier transform with respect to $x$ and $v$ is given by
\[ \hat K_s(t,\psi,\xi) = c \exp\left(-\int_0^t |\xi + \tau \psi|^{2s} \dd \tau \right). \]
It is easy to see that $K_s$ is smooth in all variables, so the fractional Kolmogorov equation enjoys analogous regularization properties as the classical one. It is also possible to see that the kernels $K_s$ are bounded, nonnegative and integrable. They have some polynomial decay as $x$ and $v$ go to infinity. However, the exact asymptotics have not been precisely described yet.

The regularity analysis for nonlinear equations depends on regularity results for \emph{linear} equations with variable (possibly rough) coefficients. In this section, we describe the three most fundamental regularity results, which bring to the kinetic setting the theorems of De Giogi-Nash-Moser and Schauder.

\subsection{De Giorgi meets kinetic equations}

The theorem of De Giorgi, Nash and Moser is arguably the most fundamental regularity result for nonlinear elliptic PDE. It is usually stated in terms of \emph{linear} equations. But since it does not have any regularity assumption on the coefficients, it is readily applicable to solutions of quasilinear elliptic equations whose coefficients depend somehow on the solution.

For kinetic equations with classical second order diffusion in divergence form, we study equations of the following form.
\begin{equation} \label{e:kinetic-div-form}
f_t + v \cdot \nabla_x f - \partial_{v_i} a_{ij}(t,x,v) \partial_{v_j} f = b \cdot \nabla_v f + h.
\end{equation}
Here, the function $h$ is a source term. We also included an extra drift term $b \cdot \nabla_v f$ on the right hand side. The coefficients $a_{ij}(t,x,v)$ are assumed to be uniformly elliptic. It means that there exists constants $\Lambda \geq \lambda > 0$ so that for all $(t,x,v)$ in the domain of the equation,  $\lambda \I \leq \{a_{ij}(t,x,v)\} \leq \Lambda \I$. It is important that there is no regularity assumption on the coefficients $a_{ij}$ beyond boundedness and measurability.

The kinetic version of De Giorgi, Nash, Moser theorem provides H\"older continuity estimates for solutions to \eqref{e:kinetic-div-form}. When the right hand side $h$ vanishes, there is also some form of Harnack inequality that holds. The following theorem was developed in \cite{pascucci2004}, \cite{zhang2008}, \cite{wang2009}, \cite{wang2011} and \cite{gimv-harnacklandau}. The presentation below matches the result in the last one of these papers.

\begin{thm} \label{t:kinetic-DG-div}
Assume $f: [0,1] \times B_1 \times B_1 \to \R$ is a weak solution of \eqref{e:kinetic-div-form} in $(0,1] \times B_1 \times B_1$ for some bounded functions $b$ and $h$. Then, there exists an $\alpha>0$ depending only on dimension and the ellipticity parameters $\lambda$ and $\Lambda$, and a constant $C$ depending on dimension, $\lambda$, $\Lambda$ and $\|b\|_{L^\infty}$, so that 
\[ \|f\|_{C^\alpha((1/2,1)\times B_{1/2} \times B_{1/2})} \leq C \left( \|f\|_{L^\infty([0,1] \times B_1 \times B_1)} + \|h\|_{L^\infty([0,1] \times B_1 \times B_1)} \right).\]
\end{thm}

Theorem \ref{t:kinetic-DG-div} can be used to conclude that solutions to the inhomogeneous Landau equation that satisfy \eqref{e:mass_density_assumption}, \eqref{e:energy_density_assumption} and \eqref{e:entropy_density_assumption} are H\"older continuous. It is a key ingredient in the proof of Theorem \ref{t:conditional-regularity_landau}.

In order to study the regularity of the Boltzmann equation, we need a similar result but for integro-differential diffusion. Let us now consider equations of the form
\begin{equation} \label{e:kinetic-integral}
f_t + v \cdot \nabla_x f - \int_{\R^d} (f(t,x,v') - f(t,x,v)) K(t,x,v,v') \dd v' = h.
\end{equation}
Here $h$ is a source term. The kernel $K$ is a nonnegative function. We want to study equation \eqref{e:kinetic-integral} with an eye on possible applications to the Boltzmann equation. Thus, we will assume that the kernel $K$ satisfies the nondegeneracy conditions described in section \ref{s:ide}.

The following theorem is proved in \cite{imbert-silvestre-whi2020}.

\begin{thm} \label{t:kinetic-DG-integral}
Assume $f: [0,1] \times B_1 \times \R^d \to \R$ is a bounded weak solution of \eqref{e:kinetic-integral} for $(t,x,v) \in (0,1] \times B_1 \times B_1$ and some bounded function $h$. We make the following assumptions on the kernel $K$.
\begin{itemize}
\item \textbf{Upper bound on average}. There is a constant $\Lambda$ such that for all $(t,x,v) \in (0,1] \times B_1 \times B_1$ and $r>0$,
\[ \int_{B_{r}(v)} K(t,x,v,v') |v-v'|^2 \dd v' \leq \Lambda r^{2-2s}.\]
\item \textbf{Nondegeneracy condition}. There exist two constants $\lambda>0$ and $\mu>0$ so that for any $(t,x,v) \in (0,1] \times B_1 \times B_1$ and any ball $B \subset B_2$ that contains $v$, we have 
\[ |\{v' \in B : K(t,x,v,v') \geq \lambda |v-v'|^{-d-2s} \}| \geq \mu |B|. \]
\item \textbf{The cancellation condition}.  For all $(t,x,v) \in (0,1] \times B_1 \times B_1$ and $r \in [0,1]$,
\[ \left\vert \int_{B_r(v)} K(t,x,v,v') - K(t,x,v',v) \dd v' \right\vert \leq \Lambda r^{-2s}.\]
If $s \geq 1/2$, we also have
\[ \left\vert \int_{B_r(v)} (v-v') (K(t,x,v,v') - K(t,x,v',v)) \dd v' \right\vert \leq \Lambda r^{1-2s}.\]
\end{itemize}
Then $f$ is H\"older continuous and it satisfies the following estimate for some $\alpha>0$.
\[ \|f\|_{C^\alpha((1/2,1)\times B_{1/2} \times B_{1/2})} \leq C \left( \|f\|_{L^\infty([0,1] \times B_1 \times \R^d)} + \|h\|_{L^\infty([0,1] \times B_1 \times B_1)} \right).\]
The constants $\alpha$ and $C$ depend on dimension, $\lambda$, $\Lambda$ and $\mu$ only.
\end{thm}

Note that the assumptions on Theorem \ref{t:kinetic-DG-integral} are exactly the union of the assumptions of Proposition \ref{p:Hs_boundedness} and Theorem \ref{t:coercivity}. The Boltzmann kernel satisfies a cone of nondegeneracy condition that implies the nondegeneracy hypothesis of Theorem \ref{t:kinetic-DG-integral}.

Remarkably, Theorems \ref{t:kinetic-DG-div} and \ref{t:kinetic-DG-integral} are not specifically for the Landau and Boltzmann equation as in \eqref{e:boltzmann}. They apply to a more general family of kinetic Fokker-Planck equations. In the case of Landau equation, the coefficients $a_{ij}$ will be related to the solution through \eqref{e:landau}. In the case of the Boltzmann equation, the kernel $K$ will be related to $f$ through \eqref{e:boltzmann-kernel}. Theorems \ref{t:kinetic-DG-div} and \ref{t:kinetic-DG-integral} apply regardless of these extra relations.

\subsection{Schauder meets kinetic equations}

The Schauder estimates provide an a priori estimate in higher order H\"older spaces when we know, in addition to uniform ellipticity, that the coefficients of the equation are H\"older continuous.

Unlike the H\"older estimate obtained using De Giorgi's method, the Schauder estimates provide an estimate in a H\"older norm with a precise exponent. When we consider a kinetic Fokker-Planck equation like \eqref{e:kinetic-div-form}, we gain exactly two derivatives in $v$. In the fractional order case \eqref{e:kinetic-integral}, we should expect to gain $2s$ derivatives in $v$. The exact gain of regularity with respect to the variables $t$ and $x$ is less obvious. We need to keep track of the precise balance of scales between these variables. For that, it is important to analyze the group of transformations that leave the class of equations \eqref{e:kinetic-div-form} or \eqref{e:kinetic-integral} invariant.

Let us focus on the fractional order case \eqref{e:kinetic-integral}, which is the more complex of the two. We observe that if $f$ solves \eqref{e:kinetic-integral} for some kernel $K$ satisfying the hypothesis of Theorem \ref{t:kinetic-DG-integral}, then the same is true for the scaled function
\[ f_r(t,x,v) = f(r^{2s}t, r^{1+2s}x, rv). \]
The kernel $K$ would have to be replaced with $r^{d+2s} K(r^{2s}t, r^{1+2s}x, rv)$, which satisfies the same assumptions as the original $K$.

Because this is the natural scaling of the equation, we define the kinetic cylinders at the origin $Q_r$ as
\[ Q_r:= (-r^{2s},0] \times B_{r^{1+2s}} \times B_r.\]

The class of equations \eqref{e:kinetic-integral} is not translation invariant because the coefficient in the drift term depends on $v$. It is \emph{Galilean} invariant, in the sense that we can change coordinates to another inertial frame of reference. That is, given $(s,y,w)$, if we define
\[ (s,y,w) \circ (t,x,v) = (s+t, y+x+tw, w+v),\]
then $\tilde f(t,x,v) = f((s,y,w) \circ (t,x,v))$ would satisfy also an equation like \eqref{e:kinetic-integral} with $K$ replaced with $\tilde K(t,x,v) = K((s,y,w) \circ (t,x,v))$.

In order to state the Schauder estimates for the equation \eqref{e:kinetic-integral}, we need to use a H\"older norm that is compatible with the scaling and Galilean invariance of the equation. The best way to achieve it is by defining an appropriate notion of distance.

\begin{defn} \label{d:distance}
For any two points $z_1 = (t_1,x_1,v_1)$ and $z_2 = (t_2, x_2, v_2)$ in $\R^{1+2d}$, we define the following \emph{kinetic} distance
\[ d_\ell(z_1,z_2) := \min_{w \in \R^d} \left\{ \max \left( |t_1-t_2|^{ \frac 1 {2s} } , |x_1-x_2-(t_1-t_2)w|^{ \frac 1 {1+2s} } , |v_1-w| , |v_2-w| \right) \right\}.\]
\end{defn}

Strictly speaking, this is only a distance (it satisfies the triangle inequality) when $s \geq 1/2$. For smaller values of $s \in (0,1/2)$, $d_\ell^{2s}$ is a distance. In either case, we use the expression for $d_\ell$ as in Definition \ref{d:distance} to keep consistency of the $1$-homogeneity with respect to the variable $v$.

The distance $d_\ell$ has the following two properties.
\begin{itemize}
  \item It is homogeneous with respect to the scaling of the equation, in the sense that
  \[ d_\ell((r^{2s}t_1, r^{1+2s}x_1, rv_1),(r^{2s}t_2, r^{1+2s}x_2, rv_2)) = r d_\ell((t_1,x_1,v_1),(t_2,x_2,v_2)).\]
  \item It is invariant by left Galilean translations. That is, for every $z_0,z_1,z_2 \in \R^{2d+1}$,
  \[ d_\ell(z_0 \circ z_1) = d_\ell(z_0 \circ z_2).\]
\end{itemize}

The subindex ``$\ell$'' is meant to indicate that the distance is \textbf{``l''}eft invariant, as opposed to right-invariant.

In terms of this kinetic distance, we define the kinetic H\"older norms.

\begin{defn} \label{d:holder-space} Let $\Omega \subset \R^{2d+1}$. For any
  $\alpha \in (0,\infty)$, we define the $C^\alpha(z_0)$ semi-norm of a function $f : \Omega \to \R$ as the smallest constant $C$ so that for all $z \in \Omega$
  \[ |f(z) - p(z)| \leq C d_\ell(z,z_0)^\alpha,\] 
  for some polynomial $p(t,x,v)$ whose kinetic degree is less than $\alpha$.

  Moreover, the $C^\alpha$ semi-norm of $f$ in $\Omega$ is the supremum of $\|f\|_{C^\alpha(z_0)}$ for all $z_0 \in \Omega$.
  
  The full norm $\|f\|_{C_\ell^\alpha(D)}$ is defined as $[f]_{C_\ell^\alpha(D)}+[f]_{L^\infty(D)}$.
\end{defn}

We left out the definition of kinetic degree of a polynomial. It is adjusted by the scaling defined above so that the variable $v$ has degree one, $t$ has degree $2s$, and $x$ has degree $1+2s$.

Now that we have a precise notion of H\"older continuity that takes into account the scaling and Galilean invariance of the equation, we are ready to state the kinetic Schauder estimates. The following theorem is proved in \cite{imbert2018schauder}.

\begin{thm}[The Schauder estimate]
  \label{t:schauder}
  Let $\alpha \in (0, \min (1,2s))$ and $\alpha' = 2s \alpha / (1+2s)$. Let
  $K\colon Q_1 \times \R^d \to \R$ be a nonnegative kernel satisfying the following assumptions.
  \begin{itemize}
    \item \textbf{Upper bound on average}. There is a constant $\Lambda$ such that for all $(t,x,v) \in Q_1$ and $r>0$,
\[ \int_{B_{r}(v)} K(t,x,v,v') |v-v'|^2 \dd v' \leq \Lambda r^{2-2s}.\]
    \item \textbf{Nondegeneracy condition}. There exist two constants $\lambda>0$ and $\mu>0$ so that for any $(t,x,v) \in Q_1$ and any ball $B \subset B_2$ that contains $v$, we have 
\[ |\{v' \in B : K(t,x,v,v') \geq \lambda |v-v'|^{-d-2s} \}| \geq \mu |B|. \]
  \item \textbf{Nondivergence symmetry}.  For all $(t,x,v) \in Q_1$ and $w \in \R^d$, $K(t,x,v,v+w) = K(t,x,v,v-w)$.
  \item \textbf{Holder continuity}. For all $(t_1,x_1,v_1), (t_2,x_2,v_2) \in Q_1$ and $r>0$,
  \[ \int_{B_r} |K(t_1,x_1,v_1,v_1+w)-K(t_2,x_2,v_2,v_2+w)| |w|^2 \dd w \leq A_0 r^{2-2s} d_\ell((t_1,x_1,v_1),(t_2,x_2,v_2))^{\alpha'}. \]
  \end{itemize}

  Let $h: Q_1 \to \R$ be $\alpha'$-H\"older continuous. Assume further that $2s+\alpha' \notin \{1,2\}$.

  If $f$ satisfies \eqref{e:kinetic-integral} in $Q_1$, then
  \[
    [f]_{C^{2s+\alpha'} (Q_{1/2})} \le C (\| f \|_{C^\alpha ((-1,0] \times B_1 \times \R^d)} + \|h\|_{C^{\alpha'} (Q_1)}) .
  \]
  The constant $C$ only depends on dimension, $\alpha$, the order $2s$ of the
  integral diffusion, ellipticity constants $\mu,\lambda,\Lambda$ and
  $A_0$.
\end{thm}

The first two conditions on the kernel are the same as in Theorem \ref{t:kinetic-DG-integral}. The symmetry condition replaces the cancellation condition. Schauder theorem is, in this case, a statement for equations in nondivergence form. The last condition on the kernel is the integral-version, with respect to the kinetic distance, of the H\"older continuity of the coefficients in the classical case.

The Schauder estimates for kinetic equations with standard second order diffusion is a particular case of a more general theory developed in the context of hypoelliptic equations (see \cite{sergio2004recent} and references therein). A simplified version of the proof in \cite{imbert2018schauder} also leads to the following result.

\begin{thm}[The Schauder estimate for kinetic equations with second order diffusion]
  \label{t:schauder-2nd-order}
  Let $\alpha \in (0, 1)$. Let $h: Q_1 \to \R$ be $\alpha$-H\"older continuous. 

  Suppose $f$ is a classical solution of
  \[ f_t + v \cdot \nabla_x f - a_{ij}(t,x,v) \partial_{v_i v_j} f = h \text{ in } Q_1,\]
  for some $C^\alpha$ function $h$ and uniformly elliptic coefficients $a_{ij}$ that are $C^\alpha$ H\"older continuous in $Q_1$ with respect to the $d_\ell$ distance defined with $s=1$. Then
  \[
    [f]_{C^{2+\alpha} (Q_{1/2})} \le C (\| f \|_{C^\alpha (Q_1)} + \|h\|_{C^{\alpha} (Q_1)}) .
  \]
  The constant $C$ depends on dimension, ellipticity constants $\lambda,\Lambda$ and the H\"older norm of $a_{ij}$.
\end{thm}

\section{Proving the conditional regularity estimates of Theorem \ref{t:conditional-regularity_boltzmann}}

Let us outline the ingredients in the proof of Theorem \ref{t:conditional-regularity_boltzmann}. It has the following steps.

\begin{enumerate}
  \item Obtain pointwise upper bounds for the solution $f$ depending on the constants in \eqref{e:mass_density_assumption}, \eqref{e:energy_density_assumption} and \eqref{e:entropy_density_assumption}. These upper bounds should decay faster than $\langle v \rangle^{-q}$, for any exponent $q>0$.
  \item Use Theorem \ref{t:kinetic-DG-integral} to deduce H\"older estimates for $f$.
  \item Once we know that $f$ is H\"older, we see through the formula \eqref{e:boltzmann-kernel} that the kernel $K_f$ is H\"older continuous as well. Thus, we use the Schuader estimates of Theorem \ref{t:schauder} to achieve higher order H\"older estimates for $f$.
  \item There is a change of variables that makes the ellipticity of $K_f$ uniform for large velocities. In that way, the local regularity estimates given by Theorems \ref{t:kinetic-DG-integral} and \ref{t:schauder} are transformed into global estimates.
  \item We iterate the application of the Schauder estimates to derivatives of the function $f$ to obtain $C^\infty$ estimates.
\end{enumerate}

While Theorems \ref{t:kinetic-DG-integral} and \ref{t:schauder} are results about generic kinetic equations with integral diffusion, the upper bounds and the change of variables are two steps that are specific about the Boltzmann equation. We describe these two ingredients below. A more detailed description of the ideas in the proof of Theorem \ref{t:conditional-regularity_boltzmann} can be found in \cite{imbert-silvestre-survey2020}.

\subsection{The pointwise upper bounds}

The following theorem is proved in \cite{imbert-mouhot-silvestre-decay2020}.

\begin{thm} \label{t:upper-bounds}
Let $f$ be a (classical) solution to the inhomogeneous Boltzmann equation \eqref{e:boltzmann}, periodic in space, where the collision operator $Q$ has the form \eqref{e:bco} and $B$ is the standard non-cutoff collision kernel of the form \eqref{e:non-cutoff} with $\gamma+2s \in [0,2]$.

Assume that there are constants $m_0>0$, $M_0$, $E_0$ and $H_0$ such that the hydrodynamic bounds \eqref{e:mass_density_assumption}, \eqref{e:energy_density_assumption} and \eqref{e:entropy_density_assumption} hold. Then
\begin{itemize}
\item If $\gamma > 0$, for $t>0$ we get $f \leq C_q \langle v \rangle^{-q}$ for all $q>0$. This is a self-generated bound independent of the initial data.
\item If $\gamma \leq 0$, and we assume in addition $f_0 \leq A_q \langle v \rangle^{-q}$ for all $q>0$, then we have
$f_0 \leq C_q \langle v \rangle^{-q}$.
\end{itemize}
The constants $C_q$ depend on dimension, $\gamma$, $s$, $m_0$, $M_0$, $E_0$, $H_0$, and, in the case $\gamma \leq 0$, also on the constants $A_q$.
\end{thm}

The proof of Theorem \ref{t:upper-bounds} is based on some type of nonlocal barrier argument. We argue by contradiction. We look at the first point where the function invalidates our pointwise upper bound, and somehow a contradiction is reached by carefully analyzing the equation at that point. Using similar techniques, we also obtain Maxwellian lower bounds in \cite{imbert-mouhot-silvestre-lowerbound2020}.

The first time a barrier argument was used in the context of the Boltzmann equation was in \cite{gamba2009} to obtain Maxwellian upper bounds in the case of the space homogeneous cutoff Boltzmann. See also \cite{taskovic,alonso2019} for related results for the space-homogeneous non-cutoff Boltzmann.

These pointwise upper bounds tell us that the solution $f$ decays as $|v| \to \infty$ faster than any algebraic rate. We may compare them with the classical moment estimates for the Boltzmann equation. Needless to say, pointwise estimates are a much stronger result than moment estimates. Theorem \ref{t:upper-bounds} implies in particular that for every $p>0$ there is a $C_p$ so that for all $(t,x) \in [0,\infty) \times \R^d$,
\[ \int_{\R^d} f(t,x,v) \langle v \rangle^p \dd v \leq C_p.\]
Moment estimates like these are well known in the space-homogeneous case. However, there is no direct proof in the space inhomogeneous case, other than as a consequence of the pointwise estimates.

\subsection{A change of variables to handle large velocities}

Theorems \ref{t:kinetic-DG-integral} and \ref{t:schauder} are used to obtain estimates in H\"older norms for the solution $f$ of the Boltzmann equation \eqref{e:boltzmann}. The hydrodynamic assumptions \eqref{e:mass_density_assumption}, \eqref{e:energy_density_assumption} and \eqref{e:entropy_density_assumption} are used to ensure that the kernel $K_f$ satisfies the appropriate nondegeneracy conditions. However, a careful analysis of the formula \eqref{e:boltzmann-kernel} reveals that while the three assumptions of Theorem \ref{t:kinetic-DG-integral} hold in any bounded set of velocities $v$, its parameters degenerate as $|v| \to \infty$. 

In order to obtain global estimates in H\"older spaces, we must figure out what happens to the kernels $K_f$ as $|v| \to \infty$. The following change of variables resolves this problem altogether.

For any value of $v_0 \notin B_1$, we define the linear transformation $T_{v_0}$ by the formula
\[
  T_{v_0}(a v_0 + w) := \frac{a}{|v_0|} v_0 + w \qquad \text{  whenever } w \perp v_0.
\]
Note that $T_{v_0}$ maps the unit ball $B_1$ into an ellipsoid that is flattened by the factor $1/|v_0|$ in the direction of $v_0$. If $v_0 \in B_1$, we simply take $T_{v_0}$ to be the identity. Given any $z_0 = (t_0,x_0,v_0)$, we further define
\[
  \mathcal T_{z_0} := z_0 \circ \left(|v_0|^{-\gamma-2s} t,
    |v_0|^{-\gamma-2s} T_{v_0} x, T_{v_0} v \right).
\]
Here, $\circ$ is the Galilean group operator in $\R^{1+2d}$. This transformation $\mathcal T_{z_0}$ maps $Q_1$ into certain neighborhood of $z_0$ in $\R^{2d+1}$.

Let $f$ be a solution to the non-cutoff Boltzmann equation. Consider $\bar f (z)=f (\mathcal T_{z_0} (z))$. This function $\bar f$ solves a modified equation in $Q_1$
\[ \partial_t \bar f + v \cdot \nabla_x \bar f - \int_{\R^d} (f(v') - f(v)) \bar K_f(t,x,v,v') \dd v' = \bar h,\]
where
\[ \bar h(z) = c |v_0|^{-\gamma-2s} \bar f(z) (f \ast_v |\cdot|^\gamma)(\mathcal T_{z_0} z),\]
and
\[ \bar K_f(t,x,v,v') = |v_0|^{-\gamma-2s} K(\mathcal T_{z_0} z, v_0 + T_{v_0} v').\]

Remarkably, the kernel $\bar K_f$ satisfies the assumptions of Theorem \ref{t:kinetic-DG-integral} and \ref{t:schauder} with parameters that are independent of $v_0$. Thus, by working out how the H\"older norms are affected by this linear change of variables, we have a recipe to compute the precise asymptotics as $|v| \to \infty$ of all our regularity estimates.

We make a precise statement in the following proposition.
\begin{prop} \label{p:change_of_variables}
Let $f : [0,T] \times \R^d \times \R^d \to [0,\infty)$ be a function that satisfies \eqref{e:mass_density_assumption}, \eqref{e:energy_density_assumption} and \eqref{e:entropy_density_assumption}. Let $K_f$ be the Boltzmann kernel that is obtained by the formula \eqref{e:boltzmann-kernel}, and $\bar K_f: Q_1 \times \R^d \to [0,\infty)$ be the modified kernel as above. Then, there are constants $\lambda>0$, $\Lambda$ and $\mu$, depending only on $m_0$, $M_0$, $E_0$ and $H_0$, but \textbf{not} on $v_0$, such that for all $(t,x,v) \in Q_1$
\begin{itemize}
\item For all $r>0$,
\[ \int_{B_{r}(v)} K(t,x,v,v') |v-v'|^2 \dd v' \leq \Lambda r^{2-2s}.\]
\item For any ball $B \subset \R^d$ that contains $v$, we have 
\[ |\{v' \in B : K(v,v') \geq \lambda |v-v'|^{-d-2s} \}| \geq \mu r^d. \]
\item For all $r \in [0,1]$,
\[ \left\vert \int_{B_r(v)} K(t,x,v,v') - K(t,x,v',v) \dd v' \right\vert \leq \Lambda r^{-2s}.\]
If $s \geq 1/2$, we also have
\[ \left\vert \int_{B_r(v)} (v-v') (K(t,x,v,v') - K(t,x,v',v)) \dd v' \right\vert \leq \Lambda r^{1-2s}.\]
\end{itemize}
\end{prop}

Proposition \ref{p:change_of_variables} tells us that the assumptions of Theorem \ref{t:kinetic-DG-integral} hold uniformly after the change of variables. A similar computation holds for the H\"older continuity assumption of Theorem \ref{t:schauder}. Effectively, Proposition \ref{p:change_of_variables} turns the local estimates of Theorems \ref{t:kinetic-DG-integral} and \ref{t:schauder} into global regularity estimates, with optimal asymptotics as $|v| \to \infty$.

Proposition \ref{p:change_of_variables} is proved in \cite{imbert2019global}. It is inspired by a similar change of variables for the Landau equation that first appeared in \cite{css-upperbounds-landau}.

\section{The homogeneous problem}

\label{s:space-homogeneous}

A simplified form of the Boltzmann or Landau equation \eqref{e:boltzmann} takes place when we consider solutions that are constant in the space variable $x$. In that case, the kinetic transport term disappears from the equation and we are left with
\begin{equation} \label{e:space-homogeneous}
f_t = Q(f,f).
\end{equation}

Here, $Q(f,f)$ may be the Boltzmann or Landau operators described in \eqref{e:bco} and \eqref{e:landau_expression1}. The equation takes the more classical form of quasilinear parabolic equations. Moreover, the conserved quantities become stronger, since they do not involve integration in $x$. More precisely, we have the following.
\begin{description}
  \item[Conservation of mass:] \[ \int_{\R^d} f(t,v) \dd v \qquad \text{ is constant in time.}\]
\item[Conservation of energy:] \[ \int_{\R^d} f(t,v) |v|^2 \dd v \qquad \text{ is constant in time.}\]
\item[Conservation of momentum:] \[ \int_{\R^d} f(t,v) v \dd v \qquad \text{ is constant in time.}\]
\item[Entropy:] \[ \int_{\R^d} f(t,v) \log f(t,v) \dd v \qquad \text{ is monotone decreasing.}\]
\end{description}

In particular, the hydrodynamic assumptions \eqref{e:mass_density_assumption}, \eqref{e:energy_density_assumption}, \eqref{e:entropy_density_assumption} are automatically true for any solution \footnote{We mean for \emph{classical} solutions with sufficient decay as $|v| \to \infty$. Some of these conservation laws might break if we work with an overly weak notion of solution.} of the space homogeneous problem whose initial data has finite mass, energy and entropy.

Establishing the existence of global smooth solutions of \eqref{e:space-homogeneous} does not face the same difficulties as in the space in-homogeneous case. Clearly, Theorems \ref{t:conditional-regularity_boltzmann} and \ref{t:conditional-regularity_landau} can be combined with any short-time well posedness result to produce a global-in-time solution. Anyhow, the global well posedness of the space homogeneous problem, at least for some range of the values $\gamma$ and $s$, can be proved directly and it was done earlier. An incomplete list of references would be \cite{desvillettes2000spatially,villani1998spatially,silvestre_landau} for the homogeneous Landau equation, and \cite{desvillettes1997,desvillettes2004,alexandre2005,huo2008,alexandre2009,desvillettes2009,morimoto2009,morimoto2010,chen2011,alexandre2012,he2012} for the homogeneous Boltzmann equation.

Regularity estimates for the space homogeneous problem face a strong limitation when $\gamma+2s<0$. This range is called the \emph{very soft potential} case. The Landau equation corresponds to $s=1$, thus very soft potentials are those with $\gamma < -2$. Note that the very soft potential range is excluded from Theorems \ref{t:conditional-regularity_boltzmann} and \ref{t:conditional-regularity_landau}.The problem is that, even in the space homogeneous scenario, we do not know how to obtain an a priori estimate for the solution in $L^\infty$ when $\gamma+2s < 0$. We explore this major open problem in the rest of this section.

In order to understand the difficulties with very soft potentials, it is best to focus on the most extreme case: the Landau-Coulomb equation. It corresponds to $s=1$ and $\gamma=-3$ in three dimensions. It is a model for plasma dynamics, what makes it particularly interesting as a natural equation from mathematical physics. The operator $Q(f,f)$ takes the following form \eqref{e:landau-coulomb}, which we recall here.
\begin{equation} \label{e:landau-coulomb2}
Q(f,f) = \bar a_{ij} \partial_{ij}  f + f^2,
\end{equation}
where 
\[ \bar a_{ij}(v) = -\partial_{ij} (-\Delta)^{-2} f(v) = \frac 1 {8\pi} \int_{\R^3} (|w|^2 \delta_{ij} - w_i w_j) |w|^{-3} f(v-w) \dd w.\]

Some lower bound on the coefficients is computed in terms of the mass, energy and entropy of $f$. Let us state it for a general power $\gamma$ in the next lemma.

\begin{lemma} \label{l:landau-coefficients}
Let $f : \R^d \to [0,\infty)$ be a function with finite mass, energy and entropy. Assume that
\begin{align*}
\int_{\R^d} f(v) \dd v &\geq m_0, \\
\int_{\R^d} f(v) |v|^2 \dd v &\leq E_0, \\
\int_{\R^d} f(v) \log f(v) \dd v &\leq H_0.
\end{align*}
Consider the coefficients $\bar a_{ij}$ given by \eqref{e:landau} for any $\gamma \in (-d-2,0]$. That is
\[ \bar a_{ij}(v) = \int_{\R^d} (|w|^2 \delta_{ij} - w_i w_j) |w|^{\gamma} f(v-w) \dd w.\]
Then, we have the following lower bounds,
\[ \bar a_{ij}(v) e_i e_j \geq c \begin{cases}
\langle v \rangle^\gamma & \text{for any } e \in S^{d-1}, \\
\langle v \rangle^{\gamma+2} & \text{if } e \cdot v = 0.
\end{cases}
\]
Here, the constand $c$ depends on $\gamma$, dimension, $m_0$, $E_0$ and $H_0$ only.
\end{lemma}

The proof of Lemma \ref{l:landau-coefficients} can be found in \cite{silvestre_landau} (see the proof of Lemma 3.1 inside that paper). See \cite[Lemma 2.1]{css-upperbounds-landau} for the corresponding upper bounds for $\gamma \in [-2,0]$.

If $\gamma < -2$, there is no pointwise upper bound for $\bar a_{ij}$ in terms of the mass, energy and entropy of $f$ only. We may see this fact as a hint that analyzing the very soft potential case $\gamma<-2$ will be delicate. In order to obtain a pointwise upper bounds for the coefficients $\bar a_{ij}$, we would need to have some extra information about the function $f$. For example, it would suffice to have an upper bound of its weighted $L^p$ norm
\[ \|f\|_{L^p_k(\R^d)} = \left( \int |f(v)|^p \langle v \rangle^k \dd v \right)^{1/p} \leq K.\]
In that case, if $p > d/(d+2+\gamma)$ and $k$ is sufficiently large (depending on $p$, $\gamma$ and dimension), the upper bounds for the coefficients $\bar a_{ij}$ would differ from the lower bounds by a constant factor. A conditional a priori estimate for the solution to \eqref{e:space-homogeneous} follows in that case (see \cite{silvestre_landau}).

The bounds for the coefficients $\bar a_{ij}$ might suggest a comparison between the homogeneous Landau-Coulomb equation and the nonlinear heat equation $f_t = \Delta f + f^2$. Since the latter blows up in finite time, one may be inclined to believe that the Landau-Coulomb equation also develops singularities. It is an example of how we may arrive to the wrong conclusions when we think of the collision operator $Q(f,f)$ as if it was semilinear \footnote{It would be equally misleading to think of the Boltzmann collision operator as the fractional Laplacian plus lower order terms.}. The coefficients $\bar a_{ij}$ are not bounded above for $\gamma<-2$. Intuitively, the coefficients $\bar a_{ij}$ will become large around the areas where the mass of $f$ is accumulating. This enhanced dissipation seems to prevent blowup altogether. Nowadays, the consensus among researchers in the area is that everybody expects the homogeneous Landau-Coulomb equation to have global smooth solutions for any reasonable initial data.

There are two positive, unconditional, results about the Landau-Coulomb equation that provide some partial progress in the problem. They are the following.

\begin{enumerate}
  \item It is shown in \cite{desvillettes2014entropy} that we can get a bound for $\|f\|_{L^1_t L^3_{-k}}$, for some suitable exponent $k$, using the entropy dissipation.
  \item In \cite{golse2019partial}, the authors show that the potential set of times where a weak solution blows up has Hausdorff dimension at most $1/2$.
\end{enumerate}

There are also conditional regularity estimates for the homogeneous Landau equation, which say that the solution will not blow up unless certain quantities become infinite. Results of that kind can be found in \cite{silvestre_landau,gualdani2016,gualdani2019}.

\subsection{The Krieger-Strain model}

A simplified toy model that captures the properties and key difficulties of the Landau-Coulomb operator was proposed by Krieger and Strain in \cite{krieger2012a}. The idea is to replace the Landau-Coulomb operator \eqref{e:landau-coulomb2} with
\[ Q(f,f) = (-\Delta)^{-1} f(v) \cdot \Delta f(v) + f(v)^2.\]
Solutions to the equation \eqref{e:space-homogeneous} with this new operator $Q$ still conserve mass and dissipate entropy, but they do not conserve energy. Note also that the Maxwellian functions are no longer stationary solutions. In fact, the only function $f$ so that $Q(f,f)=0$ is $f \equiv 0$.

We can modify the operator further by introducing a parameter $\alpha \in [0,1]$ and letting
\[ Q_\alpha(f,f) = (-\Delta)^{-1} f(v) \cdot \Delta f(v) + \alpha f(v)^2.\]
The idea is that by setting $\alpha<1$, we are making the dissipation term stronger than the reaction term. In \cite{krieger2012a}, the authors study the Cauchy problem for \eqref{e:space-homogeneous} with $Q_\alpha$, and establish the existence of global smooth solutions for radially symmetric, positive and monotone initial data, and $\alpha < 2/3$. Soon after, in \cite{krieger2012b}, the result is extended to the range $\alpha \in (0,74/75)$. The milestone $\alpha=1$ is achieved in \cite{gualdani2016}, also for symmetric, positive and radial initial data.

Note that all the regularity estimates for the Krieger-Strain equation in the current literature are limited to radially symmetric solutions.

\subsection{Toward global smooth solutions}

We mentioned above that we believe that the space-homogeneous equation \eqref{e:space-homogeneous}, with the Landau-Coulomb operator \eqref{e:landau-coulomb2} has global smooth solutions for any reasonably smooth initial data that decays quickly as $|v| \to \infty$. But how could we ever possibly prove it? In this subsection, we outline some potential ideas to attack this problem. This is a purely speculative section. We state two conjectures. We also state two lemmas, with complete proofs, that support possible directions to study Conjectures \ref{c:Linfty-growth} and \ref{c:liouville}.

If we had an upper bound in $L^\infty$ for the solution $f$ of \eqref{e:space-homogeneous}, further regularity would follow by a standard application of the Schauder estimates. The key difficulty is to determine, given a (classical) solution $f : [0,T] \times \R^d \to [0,\infty)$ of \eqref{e:space-homogeneous}, whether $\|f(t,\cdot)\|_{L^\infty}$ may blow up at time $T$ or not.

Given mass, momentum and energy, the Maxwellian distribution is the function that minimizes the entropy. It is also easy to verify that the entropy dissipation of a nonnegative function $f$ is zero if and only if $f$ is a Maxwellian. Thus, a global solution to \eqref{e:space-homogeneous} will naturally converge to the unique Maxwellian function $M(v)$ whose mass, momentum and energy coincide with those of the initial data $f_0$. Before the limit $t \to \infty$ can effectively take place, we need to make sure that no singularity emerges from the flow of the equation. If we had a solution $f : [0,T] \times \R^d \to [0,\infty)$ so that $\lim_{t \to T} f(t,x_t) = +\infty$, along some curve $x_t$ we should also have the $L^p$ norm of $f$ blowing up around the points $x_t$, for any $p > 3/2$. In that case, it is reasonable to expect some separation of scales: to have some local blow-up profile at a small scale, separated from the bulk of the mass. The Landau-Coulomb operator \eqref{e:landau-coulomb2} is nonlocal only though the formula of the coefficients $\bar a_{ij}$. The values of $f(t,y)$, for $y$ far from $x_t$, would only influence the equation around $x_t$ by enhancing its dissipation. Intuitively, more dissipation reduces the chances of a blow-up. The more negative $\gamma$ is in \eqref{e:landau}, the more localized the formula for $\bar a_{ij}$ is in terms of $f$. In the very soft potential case, and in particular for Landau-Coulomb, we might expect that the local blow-up profile solves its own Landau equation \eqref{e:landau}, but at accelerated time scale. We will do a more precise blow-up analysis below, that justifies this intuition in part. If the local blow-up profile also solves \eqref{e:space-homogeneous}, then it should also converge to a Maxwellian. The mass, momentum and energy of the local blow-up profile are completely uncorrelated with the global function $f$ (and we will see that they may even be infinite). The following proposition shows an interesting inequality that is independent of their values.

\begin{prop} \label{p:Linfty-growth}
Let $m>0$, $q \in \R^d$ and $e > 0$ be given. Let $M : \R^d \to [0,\infty)$ be the Maxwellian distribution so that
\[ m = \int_{\R^d} M(v) \dd v, \qquad  q = \int_{\R^d} M(v) v \dd v, \qquad e = \int_{\R^d} M(v) |v|^2 \dd v.\]
For any nonnegative function $f \in L^1_2(\R^d)$ so that
\[ m = \int_{\R^d} f(v) \dd v, \qquad  q = \int_{\R^d} f(v) v \dd v, \qquad e = \int_{\R^d} f(v) |v|^2 \dd v,\]
we have \[ \frac{\|M\|_{L^\infty(\R^d)}} {\|f\|_{L^\infty(\R^d)}} \leq C_d,\]
for some constant $C_d$ depending on dimension only. Moreover, the equality is achieved if and only if $f$ is the characteristic function of some ball.
\end{prop}

We can compute the constant $C_d$ explicitly for each dimension. In three dimensions, we have
\[ C_3 = \sqrt{\frac{250}{9 \pi}}.\]

\begin{proof}
By translating and scaling the functions $f$ and $M$ if necessary, we assume without loss of generality that $q=0$, $m=1$ and $\|f\|_{L^\infty}=1$. We analyze, under these conditions, what the maximum possible value of $\|M\|_{L^\infty}$ is.

The Maxwellian $M$ is uniquely determined by the values of $m$, $q$ and $e$. In this case, after fixing $q=0$ and $m=1$, we observe that $\|M\|_{L^\infty}$ will be maximal if $e$ takes the minimum possible value. Indeed, if $M_1$ is the Maxwellian with $q=0$, $m=1$, and $e=1$, we scale it to recover a Maxwellian $M$ with any arbitrary energy $e$ by
\[ M(v) = e^{-d/2} M_1(e^{-1/2} v).\]
Thus, the problem is reduced to minimize the energy of the function $f$ constrained to $0 \leq f \leq 1$, having unit mass and zero momentum. 

Let $B_R$ be the ball centerd at the origin with volume one, and $\chi(v)$ be its characteristic function. We claim that the energy of $f$ is always larger or equal to the energy of $\chi$. We compute
\begin{align*}
\int f(v) |v|^2 \dd v  - \int \chi(v) |v|^2 \dd v &= \int_{B_R} (f(v)-\chi(v)) |v|^2 \dd v + \int_{\R^d \setminus B_R} (f(v)-\chi(v)) |v|^2 \dd v\\
\intertext{Since $f$ takes values in $[0,1]$, then $f(v)-\chi(v)$ is nonpositive in $B_R$ and nonnegative in $\R^d \setminus B_R$. In both cases, we obtain an inequality in the same direction by replacing $|v|^2$ with $R^2$.}
&\leq \int_{B_R} (f(v)-\chi(v)) R^2 \dd v + \int_{\R^d \setminus B_R} (f(v)-\chi(v)) R^2 \dd v \\
&= R^2 \int_{\R^d} (f(v)-\chi(v)) \dd v = 0
\end{align*}
The last equality is a consequence of $f$ and $\chi$ having the same mass.

The inequality will be strict unless $f = \chi$ a.e.
\end{proof}

Proposition \ref{p:Linfty-growth} says that the ratio between the maximum of the initial data $f_0$, and the maximum of the asymptotic limit of $f(t,\cdot)$ as $t \to \infty$, are bounded by a universal constant depending on dimension only. Our intuition of separation of scales suggests that this ratio is the most the $L^\infty$ norm of a function should be able to grow by the evolution of the equation. We formulate the following bold conjecture.
\begin{conjecture} \label{c:Linfty-growth}
Let $f: [0,T] \times \R^3 \to [0,\infty)$ be a classical solution of \eqref{e:space-homogeneous}, where $Q$ is the Landau-Coulomb operator \eqref{e:landau-coulomb2}, and $f(0,v) = f_0(v)$. Then, for all values of $(t,v) \in [0,T] \times \R^3$, we have the inequality
\[ f(t,v) \leq \sqrt{\frac{250}{9 \pi}} \max f_0(v).\]
\end{conjecture}

Conjecture \ref{c:Linfty-growth}, if true, rules out the finite time blow-up for solutions to \eqref{e:space-homogeneous} in the Landau-Coulomb case. Moreover, since the inequality is independent of the initial mass, momentum and energy, it suggests that the required analysis should be oblivious of these conserved quantities.

While we formulated Conjecture \ref{c:Linfty-growth} for the Landau-Coulomb case, similar intuition applies for solutions to \eqref{e:space-homogeneous} where $Q$ is the Boltzmann or Landau operator in the soft potential case. The intuition is less clear for hard potentials, but even in that case there is no apparent counterexample to rule out this inequality.

The analysis of blow-up limits is a common tool in the analysis of regularity of solutions across partial differential equations. Let us describe some attempt to apply that logic to the homogeneous Landau-Coulomb equation. Our objective is to prove that no solution to the Landau-Coulomb equation may blow up in finite time. Let us consider the simpler scenario of a radially symmetric solution that is monotone decreasing along the radial direction. It is a class of functions that is preserved by the evolution, and they can only possibly blow up at the origin.

We argue by contradiction. Let us suppose that the radially symmetric, monotone solution, $f$ blows up at time $T$. That is, we have a function
$f: [0,T] \times \R^3 \to [0,\infty)$, which is smooth for $t<T$ and solves \eqref{e:space-homogeneous}, where $Q$ is the Landau-Coulomb operator \eqref{e:landau-coulomb2}. We assume that $f$ depends only on $t$ and $|v|$ (it is radially symmetric), $v \cdot \nabla_v f \leq 0$ and therefore $\max f(t,\cdot) = f(t,0)$. Since it is blowing up at time $T$, we have $f(t,0) \to +\infty$ as $t \to T$. Ideally, we would want to extract a blow-up limit profile and find some contradiction.

Since $f$ attains it maximum always at $v=0$, we must have $D^2f(t,0) \leq 0$. Thus, $f_t(t,0) = Q(f,f)(t,0) \leq f(t,0)^2$. We deduce a minimum blowup rate: if $f$ is blowing up at time $t=T$, then $f(t,0) \geq (T-t)^{-1}$.

Let $t_k \to T$ be such that
\[ f(t_k,0) = \sup \{ f(t,v) : t\leq t_k, v \in \R^d\},\]
and, for some arbitrary $\eps>0$,
\begin{equation} \label{e:blowup-rate}
\left(\frac 12 + \eps \right) f(t_k,0) \geq f(t_k-f(t_k,0)^{-1},0) .
\end{equation}
The reason why we can find a sequence satisfying the second inequality is precisely the blowup rate $f(t,0) \geq (T-t)^{-1}$.

We consider the rescaled solutions
\[ f_k(t,v) = a_k^{-1} f(t_k + a_k^{-1} t, b_k v).\]
We choose $a_k = f(t_k, 0)$ and $b_k$ so that $(-\Delta_v)^{-1} f_k(0,0) = 1$. by construction, the functions $f_k$ solve the homogeneous Landau-Coulomb equation in $(-a_k t_k,0] \times \R^d$. Moreover, $0 \leq f_k \leq 1$ in the same domain, $f_k(0,0) = 1$, and $(-\Delta_v)^{-1}f_k(0,0) = 1$. Also, \eqref{e:blowup-rate} implies that
\begin{equation} \label{e:blowup-nonconstant}
f_k(-1,0) \leq \left(\frac 12 + \eps \right).
\end{equation}

Since we are working with a radially symmetric function $f$, the coefficients $\bar a_{ij}$ are isotropic at $v=0$. We have $\{\bar a_{ij}(t,0)\} = \frac 13 (-\Delta)^{-1} f(t,0) \I$. The same can be said about the rescaled functions $f_k$. The purpose of the normalization $(-\Delta)^{-1} f_k(0,0)=1$ is so that we fix the corresponding coefficients $\{\bar a_{ij}^k(0,0)\} = \frac 13 \I$.

We would like to pass to the limit as $k \to \infty$ and recover an ancient solution. We have $t_k \to T$ and $a_k \to \infty$, so the domain of the functions will grow to $(-\infty,0] \times \R^3$ as $k \to \infty$.

The following lemma says that if a radially symmetric solution of the Landau-Coulomb equation is bounded by one, and its coefficients are bounded below and above at one point, then we can obtain a bound for the coefficients in any bounded domain.

\begin{lemma} \label{l:coefficients_comparable}
Let $f : [-T,0] \times \R^3 \to [0,1]$ be a classical, radially symmetric, solution of \eqref{e:space-homogeneous}, where $Q$ is the Landau-Coulomb operator \eqref{e:landau-coulomb2}. Assume that $c_0 \leq (-\Delta_v)^{-1} f(0,0) \leq 1$ for some constant $c_0 > 0$. Then, for any values of $R>0$ and $T>0$, there are constants $C_1 \geq c_1 > 0$ so that for $(t,v) \in [-T,0] \times B_R$,
\[ c_1 \I \leq \{ \bar a_{ij}(t,v) \} \leq C_1 \I.\]
The constants $c_1$ and $C_1$ depend on $c_0$, $R$ and $T$ only.
\end{lemma}

\begin{proof}
We first estimate the value of the coefficients at $v=0$. Recall that
\[ \{ \bar a_{ij}(t,0) \} = \frac 13 (-\Delta)^{-1} f(t,0) \I = c \left( \int_{\R^3} f(t,v) |v|^{-1} \dd v \right) \I . \]

Let $r_0 > 0$ be so that $r_0^2 < c_0/2$.
Let $\varphi(v)$ be a smooth positive function so that $\varphi(v) = |v|^{-1}$ whenever $|v|>1$, and $0 < \varphi(v) \leq |v|^{-1}$ elsewhere. We define
\[ a(t) := \int_{\R^3} f(t,v) \varphi(v) \dd v.\]
Naturally, $a(t) \leq (-\Delta)^{-1}f(t,0)$. Also, using that $0 \leq f(t,v) \leq 1$, we also have $a(t) \geq (-\Delta)^{-1}f(t,0) - r_0^2$. Thus $c_0/2 \leq a(0) \leq 1$.

We compute $a'(t)$ using the integral expression for the Landau operator.
\begin{align*}
 a'(t) &= \frac 1 {8 \pi} \int \varphi(v) \partial_i \int \left( |w|^2 \delta_{ij} - w_i w_j \right) |w|^{-3} (f(t,v+w) \partial_j f(t,v) - f(t,v) \partial_j f(t,v+w)) \dd w \dd v, \\
 \intertext{We integrate by parts twice to get,}
 &= \frac 1 {8 \pi} \iint \left( \partial_{ij} \varphi(v) \left( |w|^2 \delta_{ij} - w_i w_j \right) |w|^{-3}  + 2 \partial_i \varphi(v) \frac {w_i}{|w|^3} \right) f(t,v+w) f(t,v) \dd w \dd v.
\end{align*}
Taking absolute values, we have
\[ |a'(t)| \lesssim \iint \left( \langle v \rangle^{-3} |w|^{-1}  + \langle v \rangle^{-2} |w|^{-2} \right) f(t,v+w) f(t,v) \dd w \dd v. \]

In order to estimate $|a'(t)|$, we divide the domain of integration in various subdomains. We start with $|v+w|<2\langle v \rangle$.
\begin{align*}
I &:= \iint_{|v+w|<2\langle v \rangle} \left( \langle v \rangle^{-3} |w|^{-1}  + \langle v \rangle^{-2} |w|^{-2} \right) f(t,v+w) f(t,v) \dd w \dd v, \\
\intertext{Using $f(t,v+w) \leq 1$,}
&\lesssim \int f(t,v) \left( \int_{|v+w|<2\langle v \rangle} \langle v \rangle^{-3} |w|^{-1}  + \langle v \rangle^{-2} |w|^{-2} \dd w \right) \dd v, \\
&\approx \int f(t,v) \langle v \rangle^{-1} \dd v \approx a(t)
\end{align*}

We continue with one of the terms in the integrand in the subdomain $|v+w| > 2\langle v \rangle$.
\begin{align*}
II &:= \iint_{|v+w| > 2\langle v \rangle} \langle v \rangle^{-2} |w|^{-2} f(t,v+w) f(t,v) \dd w \dd v, \\
\intertext{Using $f(t,v) \leq 1$, and changing variables $z=v+w$,}
&\lesssim \int f(t,z) \left( \int_{\langle v \rangle < |z|/2} \langle v \rangle^{-2} |v-z|^{-2} \dd v \right) \dd z, \\
&\lesssim \int f(t,z) \langle z \rangle^{-1} \dd z \approx a(t)
\end{align*}

We finish with the second term in the integrand in the subdomain $|v-w| > 2\langle v \rangle$.
\begin{align*}
II &:= \iint_{|v+w| > 2\langle v \rangle} \langle v \rangle^{-3} |w|^{-1} f(t,v+w) f(t,v) \dd w \dd v \\
\intertext{Changing variables $z=v+w$,}
&\lesssim \int f(t,z) \left( \int_{\langle v \rangle < |z|/2} \langle v \rangle^{-3} |v-z|^{-1} f(t,v) \dd v \right) \dd z \\
\intertext{Note that the domain of the inner integral is empty if $|z|<2$. Thus, $|v-z|^{-1} \lesssim \langle z \rangle^{-1}$.}
&\lesssim \int f(t,z) \langle z \rangle^{-1} \left( \int f(t,v) \langle v \rangle^{-3} \dd v \right) \dd z \\
\intertext{For an arbitrary $R>0$, we divinde the inner integral into the two subdomains $B_R$ and $\R^3 \setminus B_R$.}
&\lesssim \int f(t,z) \langle z \rangle^{-1} \left( \int_{B_R} f(t,v) \langle v \rangle^{-3}  \dd v  + \int_{\R^3 \setminus B_R} f(t,v) \langle v \rangle^{-3}  \dd v  \right) \dd z, \\
\intertext{We use $f(t,v) \leq 1$ in the first term, and the value of $a(t)$ to bound the second term.}
&\lesssim \int f(t,z) \langle z \rangle^{-1} \left( \log(1+R^3) + \langle R \rangle^{-2} a(t)  \right) \dd z. \\
\intertext{Choosing $R = \sqrt{a(t)}$, we get}
&\lesssim a(t) \left( 1+ \log(1+a(t)) \right).
\end{align*}

Adding the three terms $I$, $II$ and $III$, we conclude that $|a'(t)| \lesssim a(t) ( 1+ \log(1+a(t)))$. This ODE, together with the value of $a(0)$, implies that $a(t)$ may grow at most double exponentially as $t \to -\infty$. It has to stay bounded in any interval of time $[-T,0]$. Moreover, for small values of $a(t)$ the logarithmic correction is irrelevant. The value of $a(t)$ may decay at most exponentially as $t \to -\infty$. This ODE, together with $a(0) \geq c_0/2$, gives us a lower bound for $a(t)$ for $t \in [-T,0]$.

It is easy to obtain bounds for $(-\Delta)^{-1}f(t,0)$ from the values of $a(t)$. Indeed, $a(t) \leq (-\Delta)^{-1}f(t,0) \leq a(t) + r_0^2 \|f\|_{L^\infty}$. Thus, we have already obtained lower and upper bounds for the coefficients $\bar a_{ij}(t,0)$ for $t \in [-T,0]$. We need to extend these bounds for nonzero values of $v \in B_R$.

For the upper bounds, we observe that $\tr \bar a_{ij}(t,v) = (-\Delta)^{-1} f(t,v)$. Therefore
\begin{align*} 
\tr \bar a_{ij}(t,v) &= c \int_{\R^3} |v-w|^{-1} f(t,w) \dd w \\
&\lesssim \int_{B_r(v)} |v-w|^{-1} f(w) \dd w + \int_{\R^3 \setminus B_r(v)} |v-w|^{-1} f(w) \dd w \\
&\lesssim r^2 (\sup f) + \frac{|v|+r}r (-\Delta)^{-1} f(t,0).
\end{align*}
Choosing any arbitrary $r>0$, we get an upper bound for $\bar a_{ij}(t,v)$.

To get a lower bound for $\bar a_{ij}(t,v)$, we use that $f$ is radially symmetric in $v$. Let $\tilde f : [-T,0] \times [0,\infty) \to [0,1]$ so that $f(t,v) = \tilde f(t,|v|)$. The smallest eigenvalue $\lambda$ of $\bar a_{ij}(t,v)$ points in the radial direction. It can be computed from $\tilde f$ by the formula
\begin{align*}
\lambda(t,v) &= \frac 13 |v|^{-3} \int_0^{|v|} s^4 \tilde f(t,s) \dd s + \frac 13 \int_{|v|}^\infty s \tilde f(t,s) \dd s. \\
\intertext{We choose any $s_0 < |v|$ and get}
&\geq \frac 13 |v|^{-3} s_0^3 \int_0^\infty s \tilde f(t,s) \dd s - \frac 16 |v|^{-3} s_0^5 (\sup \tilde f) \\
&\geq \frac 13 |v|^{-3} s_0^3 (-\Delta)^{-1} f(t,0) - \frac 16 |v|^{-3} s_0^5. 
\end{align*}
This gives us a lower bound if we choose $s_0$ sufficiently small, and we conclude the proof.
\end{proof}

Going back to the blow-up sequence $f_k$, because of Lemma \ref{l:coefficients_comparable}, we have that the coefficients $a^k(t,v)$ will be uniformly elliptic in any bounded set $[-T,0] \times B_R$. Using standard parabolic estimates (first the theorem by Krylov and Safonov, and then Schauder estimates) it is easy to see that the sequence $f_k$ is uniformly smooth on bounded sets. We pass to the limit so that, up to extracting a subsequence, there is a function $f_\infty: (-\infty,0] \times \R^d \to [0,1]$ and coefficients $\bar a_{ij}^\infty : (-\infty,0] \times \R^d \to \R$, so that
\begin{align*}
f_k &\to f_\infty, \\
\partial_t f_k &\to \partial_t f_\infty, \\
D^2_v f_k &\to D^2_v f_\infty, \\
\bar a_{ij}^k &\to \bar a_{ij}^\infty,
\end{align*}
with convergence being uniform over every compact set. Clearly, the function $f_\infty$ solves
\[ \partial_t f_\infty = \bar a_{ij}^\infty \partial_{ij} f_\infty + f_\infty^2 \qquad \text{ in } (-\infty,0] \times \R^3.\]
From Fatou's lemma, we see that
\[ \bar a_{ij}^\infty = \lim_{k \to \infty} c \int_{\R^3} (|w|^2 \I - w_i w_j) |w|^{-3} f_k(v-w) \dd w \geq c \int_{\R^3} (|w|^2 \I - w_i w_j) |w|^{-3} f_\infty(v-w) \dd w.\]
We cannot say that $f_\infty$ is an ancient solution of the Landau-Coulomb equation because the coefficients $\bar a_{ij}^\infty$ are not exactly given by the formula \eqref{e:landau-coulomb2} applied to $f_\infty$. Yet, the inequality above tells us that the diffusion is only \emph{enhanced} in the limit. Intuitively, more diffusion should only improve our regularity estimates.

Admittedly, it is far from clear how to derive a contradiction from this blowup analysis. It seems reasonable to expect that the only bounded ancient solutions to the Landau-Coulomb equation are stationary Maxwellians. However, even if we start with a solution $f$ that has finite mass, energy and entropy, none of these quantities will be preserved through the sequence and blow-up limit. The final function $f_\infty$ that we obtain is a smooth, bounded function that belongs in addition to $L^\infty_{loc}((-\infty,0], L^1_{-1}(\R^d))$ (as a consequence of the upper bounds in Lemma \ref{l:coefficients_comparable}). There is no obvious reason to expect the blow-up limit $f_\infty$ to have finite mass or energy.

We state the corresponding Liouville-type conjecture here.

\begin{conjecture} \label{c:liouville}
Let $f : (-\infty,0] \times \R^3 \to [0,\infty)$ be a smooth classical solution of \eqref{e:space-homogeneous}, where $Q$ is the Landau-Coulomb operator \eqref{e:landau-coulomb2}. We make the following assumptions.
\begin{itemize}
\item For all $(t,v) \in (-\infty,0] \times \R^3$, we have $0 \leq f(t,v) \leq 1$.
\item For every $t \in (-\infty,0]$, there is some constant $C_t$ so that \[ \int_{\R^3} f(v) \langle v \rangle^{-1} \dd v \leq C_t.\]
\end{itemize}
Then $f$ must be a stationary Maxwellian.
\end{conjecture}

Note that the second assumption in Conjecture \ref{c:liouville} is necessary just to make sense of the formula for the coefficients in \eqref{e:landau-coulomb2}. In the radially symmetric case, the constants $C_t$, are related for different values of $t$ through Lemma \ref{l:coefficients_comparable}.

The blow-up limit $f_\infty$ cannot be a stationary Maxwellian because it would contradict \eqref{e:blowup-nonconstant}. That is the basic idea of how the blowup analysis would rule out the emergence of singularities. However, a positive resolution of Conjecture \ref{c:liouville} seems to be very difficult with current techniques.

It would be interesting to study solutions to the space homogeneous Landau-Coulomb equation \eqref{e:space-homogeneous}, with $Q$ given by \eqref{e:landau-coulomb2}, in the class $f \in L^\infty([0,T] \times \R^3) \cap L^\infty([0,T], L^1_{-1}(\R^3))$. These are solutions with possibly infinite mass, energy and entropy. There is currently no short-time Cauchy theory in this class.

\section{Open problems}

In Section \ref{s:space-homogeneous}, we analyzed the problem of finding $L^\infty$ bounds in the case of soft potentials, and in particular the problem of global existence of smooth solutions for the homogeneous Landau-Coulomb equation. In this section we quickly review other open problems that are related to the regularity issues reviewed in this paper.

\subsection{Unconditional regularity estimates}

Observing the conditional regularity results of Theorems \ref{t:conditional-regularity_landau} and \ref{t:conditional-regularity_boltzmann}, the obvious question is whether the hydrodynamic bounds \eqref{e:mass_density_assumption}, \eqref{e:energy_density_assumption} and \eqref{e:entropy_density_assumption} can possibly be proved.

If we somehow established the bounds for the hydrodynamic quantities in \eqref{e:mass_density_assumption}, \eqref{e:energy_density_assumption} and \eqref{e:entropy_density_assumption}, we would conclude that the inhomogeneous Landau and Boltzmann equations have global smooth solutions, for any initial data. This is a remarkable open problem that Cedric Villani describes in the first chapter of his book \cite{villani2012theoreme}. If unconditional regularity estimates were true, it seems that they would be very difficult to prove.

At certain scale, the hydrodynamic quantities associated to solutions of the Landau or Boltzmann equation approximately solve the compressible Euler equation (see \cite{bardos1991}). There are some recent results studying the stability of implosion singularities for the compressible Euler equation, and producing singularities for the compressible Navier-Stokes equation (see \cite{merle2019smooth} and \cite{merle2019implosion}). Can there be similar implosion singularities for the Boltzmann equation? The idea is to look for a solution that mostly flows toward a single point. The mass and energy densities would be blowing up at that point in a way that resembles the blowup profile of the compressible Euler equation in \cite{merle2019smooth}. This scenario appears to be more plausible than the unconditional regularity of the previous paragraph. If such an implosion singularity was indeed possible for the Boltzmann and/or Landau equation, then the bounds \eqref{e:mass_density_assumption}, \eqref{e:energy_density_assumption} and \eqref{e:entropy_density_assumption} would not be true in general.

\subsection{The kinetic Krylov-Safonov theorem}

We explain a kinetic version of De Giorgi - Nash - Moser theorem in Theorem \ref{t:kinetic-DG-div}. It concerns the equation \eqref{e:kinetic-div-form}, which is a kinetic equation with second order diffusion in divergence form.

A central result for parabolic equations in non-divergence form is the theorem of Krylov and Safonov \cite{krylov1980}. Its statement is analogous to that of De Giorgi, Nash and Moser, but for diffusion in non-divergence form. It would be natural to expect a result of that kind to hold for kinetic equations as well. It is still an open problem. We state it here as a conjecture.

\begin{conjecture} \label{c:kinetic-krylov}
Assume $f: [0,1] \times B_1 \times B_1 \to \R$ is a (classical) solution of an equation of the following form
\[ f_t + v \cdot \nabla_x f = a_{ij}(t,x,v) \partial_{v_i v_j} f.\]
Here, the coefficients $a_{ij}$ are assumed to be uniformly elliptic in the sense that there are constants $\Lambda \geq \lambda > 0$, such that for all $(t,x,v) \in [-1,0] \times B_1 \times B_1$,
\[ \lambda \I \leq \{a_{ij}(t,x,v)\} \leq \Lambda \I.\]
Then, there exists constants $\alpha>0$ and $C$, depending only on dimension and the ellipticity parameters $\lambda$ and $\Lambda$, so that 
\[ \|f\|_{C^\alpha((1/2,1)\times B_{1/2} \times B_{1/2})} \leq C \|f\|_{L^\infty([0,1] \times B_1 \times B_1)} .\]
\end{conjecture}

If would also make sense to expect an integro-differential version of Conjecture \ref{c:kinetic-krylov} to hold. One can state it for different ellipticity conditions on the kernel. As a starting point, we state the conjecture using a strong notion of fractional ellipticity.

\begin{conjecture} \label{c:kinetic-krylov-integral}
Assume $f: [0,1] \times B_1 \times \R^d \to \R$ is a (classical) solution of \eqref{e:kinetic-integral} for $(t,x,v) \in (0,1] \times B_1 \times B_1$ and some bounded function $h$. We assume that the kernel $K$  is symmetric $K(t,x,v,v+w) = K(t,x,v,v-w)$, and that there exist constants $\Lambda \geq \lambda>0$ such that
\[ \lambda |v'-v|^{-d-2s} \leq K(t,x,v,v') \leq \Lambda |v'-v|^{-d-2s}.\]
Then there are constants $\alpha>0$ and $C$, depending only on dimension, $\lambda$ and $\Lambda$, so that $f$ satisfies the estimate
\[ \|f\|_{C^\alpha((1/2,1)\times B_{1/2} \times B_{1/2})} \leq C \left( \|f\|_{L^\infty([0,1] \times B_1 \times \R^d)} + \|h\|_{L^\infty([0,1] \times B_1 \times B_1)} \right).\]
\end{conjecture}

The elliptic version of the theorem of Krylov and Safonov is easier to prove for integro-differential equations than it is for classical second order elliptic equations (see \cite{silvestre2006holder}). Based on that, it is possible that Conjecture \ref{c:kinetic-krylov-integral} may have a simpler resolution than Conjecture \ref{c:kinetic-krylov}. As of now, both remain open.

\subsection{Bounded domains}

The current versions of the conditional regularity results for Boltzmann and Landau equations in Theorems \ref{t:conditional-regularity_boltzmann} and \ref{t:conditional-regularity_landau} do not allow any form of spacial boundary.

For physical applications, it is natural to have the space variable $x$ confined to some bounded domain. There are different boundary conditions that have been considered in the literature: diffuse reflection, specular reflection and bounce back reflection.

The boundary effects may have an impact on the regularity of the solutions to kinetic equations. See \cite{guo2010decay,kim2011thesis,guo2016bv,guo2017regularity,kim2018specular,ouyang2020} for results concerning the cutoff Boltzmann and Landau equations, and solutions near a Maxwellian. It is an interesting research direction to study the possibility of extending some form of Theorem \ref{t:conditional-regularity_boltzmann} to any domain with boundary, for any of the physical boundary conditions.

In the cutoff case, it is known that the boundary conditions produce discontinuities when the domain is not convex (see \cite{kim2011thesis}). Naturally, in the non-cutoff case the solution is expected to be smooth away from the boundary. However, it is unclear if we should expect uniform smoothness estimates up to the boundary, especially in the nonconvex case.

\subsection{Sharper conditions for coercivity estimates}

In Theorem \ref{t:coercivity}, we present a sufficient condition for a kernel $K: B_2 \times B_2 \to [0,\infty)$ to have coercivity estimates with respect to the $\dot H^s$ norm. Our conditions are not sharp. We would like to explore what kernels $K$ have the property that for some $c>0$,
\[ \iint_{B_2 \times B_2} |f(v') - f(v)|^2 K(v,v') \dd v' \dd v \geq c \iint_{B_1 \times B_1} |f(v') - f(v)|^2 |v'-v|^{-d-2s} \dd v' \dd v.\]

There are simple examples replacing $K(v,v') \dd v' \dd v$ for a singular measure that satisfy the coercivity condition. For example, in two dimensions, the operator $(-\partial_1)^{2s} + (-\partial_2)^{2s}$ is clearly coercive. It corresponds to
\[ \begin{aligned} \int_{\R^2} \int_{\R} & \left( |f(v_1+w,v_2) - f(v_1,v_2)|^2 + |f(v_1,v_2+w) - f(v_1,v_2)|^2  \right) |w|^{-1-2s} \dd w \dd v \\ & \geq c \iint_{\R^d \times \R^d} |f(v') - f(v)|^2 |v'-v|^{-d-2s} \dd v' \dd v. \end{aligned}\]
What this simple example shows is that a sharp condition for coercivity should allow for singular measures instead of $K(v,v') \dd v' \dd v$, possibly supported in a set of zero Lebesgue measure.

In the context of the Boltzmann equation, Theorem \ref{t:coercivity} allows us to obtain a coercivity estimate for the Boltzmann collision operator in terms of the mass, energy and entropy of $f$. We would need a sharper version of Theorem \ref{t:coercivity} if we want to replace the upper bound on the entropy with a lower bound on the temperature tensor.

\subsection{Conditional regularity for the cutoff Boltzmann equation}

It is well known that we should not expect any regularization from the cutoff Boltzmann equation. Yet, one may start with a smooth initial data and study the propagation of its initial regularity.

It seems plausible that if the hydrodynamic bounds \eqref{e:mass_density_assumption}, \eqref{e:energy_density_assumption} and \eqref{e:entropy_density_assumption} hold, then a solution of the inhomogeneous Boltzmann equation whose initial data is smooth and rapidly decaying would stay smooth and rapidly decaying for all time. Moreover, it is conceivable that the smoothness estimates stay uniform as $t \to \infty$, just like in the noncutoff case.

Conditional regularity estimates of this type, for the full Boltzmann equation with cutoff, have not yet been studied.

\subsection{Renormalized solutions}

For generic initial data $f_0$, Alexandre and Villani constructed certain kind of global solution $f$ of the non-cutoff Boltzmann equation. This type of generalized solutions are called \emph{renormalized solutions with defect measure}. The uniqueness of solutions is not known within this class.

It is natural to wonder if a localized version of Theorem \ref{t:conditional-regularity_boltzmann} may hold for renormalized solutions with defect measure. That is, if a renormalized solution $f$ satisfies the hydrodynamic inequalities \eqref{e:mass_density_assumption}, \eqref{e:energy_density_assumption} and \eqref{e:entropy_density_assumption} for almost every $(t,x)$ in some cylinder $(-1,0] \times B_1$, should we expect $f$ to be $C^\infty$ for $(t,x) \in (-1/2,0] \times B_{1/2}$?

\subsection{Conditional regularity for solutions without rapid decay in the moderately soft potential case}

In the case of soft potentials ($\gamma\leq 0$), Theorem \ref{t:conditional-regularity_boltzmann} requires the initial data to have rapid decay. More precisely, we assume that for every $q\geq 0$, there exists a $C_q$ so that $f_0(x,v) \leq C_q \langle v \rangle^{-q}$. The reason for this assumption is that when in the proof of Theorem \ref{t:conditional-regularity_boltzmann} we iteratively compute estimates for all derivatives of $f$, the decay of our estimates deteriorate every time we take an extra derivative. If we only start with $f_0(x,v) \leq C_q \langle v \rangle^{-q}$ for one particular value of $q$, after taking certain number of derivatives, our estimates would not decay at all, and eventually the analysis fails.

It seems difficult to imagine that if $f_0(x,v) \leq C_q \langle v \rangle^{-q}$ for some large value of $q$, and the hydrodynamic bounds \eqref{e:mass_density_assumption}, \eqref{e:energy_density_assumption} and \eqref{e:entropy_density_assumption} hold, then $f$ can possibly fail to be $C^\infty$. Yet, this problem is currently open whenever $\gamma \leq 0$.

\bibliographystyle{plain}
\bibliography{barrett}
\index{Bibliography@\emph{Bibliography}}%
\end{document}